\documentclass{amsart}
\usepackage{mathrsfs}
\usepackage[all]{xy}

\usepackage{mathtools}
\usepackage{mathbbol}
\usepackage{wasysym}
\usepackage{amssymb}

\usepackage{MnSymbol}
\usepackage{comment}
\usepackage{enumerate}
\usepackage{xcolor}

\usepackage{ushort}
\usepackage{enumitem}

\usepackage{changepage}
\usepackage{algorithm}
\usepackage[noend]{algpseudocode}
\makeatletter
\def\BState{\State\hskip-\ALG@thistlm}
\makeatother

\usepackage{ifpdf}
\ifpdf 
  \usepackage[pdftex]{graphicx}
  \DeclareGraphicsExtensions{.pdf,.png,.jpg,.jpeg,.mps}
  \usepackage{pgf}
\else 
  \usepackage{graphicx}
  \DeclareGraphicsExtensions{.eps,.bmp}
  \usepackage{pgf}
  \usepackage{pstricks}
\fi
\usepackage{epic,bez123}
\usepackage{wrapfig}

\usepackage{soul}
\usepackage{url}
\usepackage{tikz}

\newtheorem{thm}{Theorem}[section]
\newtheorem{prop}[thm]{Proposition}
\newtheorem{cor}[thm]{Corollary}
\newtheorem{lem}[thm]{Lemma}

\newtheorem{claim}{Claim}

\theoremstyle{definition}
\newtheorem{defn}[thm]{Definition}

\theoremstyle{remark}

\newtheorem{rem}[thm]{Remark}

\newtheorem*{ack}{Acknowledgment}
\newcommand{\isom}{\mathrm{Isom}}

\newcommand{\U}{X}

\newcommand{\act}{\curvearrowright}

\newcommand{\f}{\mathscr F}

\newcommand{\e}[1]{\omega_{#1}}

\newcommand{\str}{\{0,1\}^{m}}

\newcommand{\ax}{\mathrm{Ax}}
\newcommand{\diam }[1]{{\textrm{diam}\big(#1\big)}}
\newcommand{\proj}{\textbf{d}}

\newcommand{\len }{\ell}

\usepackage[bookmarks=true, pdfauthor={YANG Wenyuan}]{hyperref}

\begin{document}
	\title{Growth Gaps for Quotients by Confined Subgroups}
 \author{Lihuang Ding}	
 \address{Beijing International Center for Mathematical Research\\
Peking University\\
 Beijing 100871, China
P.R.}
 \email{shanquan2@stu.pku.edu.cn}
 
 \author{Wenyuan Yang}

\address{Beijing International Center for Mathematical Research\\
Peking University\\
 Beijing 100871, China
P.R.}
\email{wyang@math.pku.edu.cn}

\thanks{(W.Y.) Partially supported by National Key R \& D Program of China (SQ2020YFA070059 and 2025YFA1017500) and  National Natural Science Foundation of China (No. 12131009 and No.12326601)}


\date{\today}

\dedicatory{} 


\maketitle

\begin{center}
{\it With an appendix by Lihuang Ding and Kairui Liu}    
\end{center}
\begin{abstract} 
In this paper, we establish growth gaps for quotients by confined subgroups in groups admitting a statistically convex-cocompact action with contracting elements. 
We also discuss applications to uniformly recurrent subgroups.
\end{abstract}

    \section{Introduction}
    In recent years, the class of confined subgroups has emerged as a significant generalization of normal subgroups and has attracted considerable research interest. Although the definition can be adapted to the setting of general topological groups, we assume throughout that $G$ is a discrete group. 
	\begin{defn}
		A nontrivial subgroup $H\subset G$ is called \textit{confined} if there exists a finite set $P\subset G$ such that for every element $g\in G$, we have $g^{-1} H g\cap (P \setminus \{1\})\neq \emptyset$. The set $P$ is called a \textit{confining subset} for $H$.
	\end{defn}
    Suppose that $G$ acts properly by isometries on a proper geodesic metric space $(X,d)$. When the action $G \act X$ is free and cocompact, confined subgroups admit a geometric characterization : a subgroup $H \leq G$ is confined if and only if the quotient space $X/H$ has uniformly bounded injectivity radius at every point. This perspective has proven particularly useful in the study of geometric convergence of Riemannian manifolds (\cite{ABBGNRS}). Assuming further that $G$ contains contracting elements, the present paper aims to investigate confined subgroups from the viewpoint of volume growth, following the theme of \cite{CGYZ}. We also discuss connections with, and applications to topics of recent interest: namely, invariant random subgroups and uniformly recurrent subgroups  in \textsection~\ref{subsec:applications}.

Let $E(G)$ denote the maximal finite normal subgroup of $G$, whose existence is guaranteed by Lemma~\ref{lem:E(G)}. Observe that the product of any subgroup with a finite normal subgroup is confined. In particular, if $E(G)$ is nontrivial, then every subgroup of $G$ becomes commensurable with a confined subgroup after multiplying by $E(G)$. Since the properties we study are invariant under commensurability, we restrict our attention to confined subgroups that admit a \textit{non-degenerate confining subset} $P \subset G$ that is disjoint from $E(G)$. This excludes those pathological examples arising from $E(G)$.

We now introduce the main growth quantities considered in this paper. 
		Fix a base point $o\in X$. For $n\ge 0$, denote $B(n) = \{g\in G\mid d(o, go)\le n\}$ . The \textit{growth rate} $\omega_Z$ of a subset $Z\subseteq G$ is defined by
		\[\omega_Z = \limsup\limits_{n\to \infty} n^{-1} \ln |B(n)\cap Z|,\]
        which is independent of the choice of $o$.
        Note that $\omega_G$ exists as a genuine limit under the assumption that $G$ contains contracting elements (see \cite[Theorem B(1)]{YANG10}). If $\omega_Z < \omega_G$, we say that $Z$ is \emph{growth tight} in $G$. A growth tight subset is also referred to as \emph{exponentially negligible}, and its complement is correspondingly called \emph{exponentially generic}.
        
		Now, let $H$ is a subgroup of $G$. Denote $B_{G/H}(n) = \{gH\mid d(o, gH o)\le n\}$. The \textit{quotient growth rate} $\omega_{G/H}$ of $H$ is defined analogously:
		\[\omega_{G/H} = \limsup\limits_{n\to \infty} n^{-1} \ln |B_{G/H}(n)|.\] 
  Throughout the paper, we assume $\omega_G<\infty$. One always has $\omega_{G/H}\le \omega_G$. A line of research initiated by Grigorchuk and de la Harpe \cite{GriH} investigates when the inequality becomes strict for infinite normal subgroups. Groups with this property are called \emph{growth tight} in the literature, and this notion has potential applications to the Hopfian property (\cite{FS23}).  

A series of research results have established growth tightness for various classes of groups with negative curvature; see \cite{GriH,AL,Sam1,YANG6,ACTao,YANG10,SZ,MGZ}. The most general result to date is the growth tightness proved for normal subgroups in \cite{ACTao} for groups satisfying the complementary growth gap property. This class of groups coincides with the statistically convex-cocompact  actions (SCC actions for short) later introduced in \cite{YANG10}, which include (relatively) hyperbolic groups, mapping class groups and CAT(0) groups with rank-one elements, among many others. SCC actions are the primary object of study in the present paper.  We refer the reader to \textsection\ref{SSecSCC} for the precise definition and  a relevant discussion of SCC actions. 
 

Generalizing these previous works, our main result establishes growth tightness for the much broader class of confined subgroups arising in SCC actions.
	\begin{thm}\label{MainThm}
		Suppose $G$ admits  a proper SCC action on a proper geodesic metric space $X$ with contracting elements. Let $P\subset G$ be a finite  non-degenerate subset. Then there exists $0<\omega_0 = \omega_0(P)< \omega_{G}$ such that for any nontrivial confined subgroup $H\subset G$ with confining subset $P$,  $$\omega_{G/H} < \omega_0.$$
	\end{thm}

Recall that $G$ contains a unique  finite normal subgroup $E(G)$, and thus $G/E(G)$ has no nontrivial finite normal subgroup.
The following  corollary is immediate.
\begin{cor}
Suppose $G$ admits  a proper SCC action on a geodesic metric space $X$ with contracting elements. Assume   $E(G)=1$.  Then for any nontrivial confined subgroup $H<G$,  $$\omega_{G/H} < \omega_G.$$   
\end{cor}  

By \cite[Theorem B(3)]{YANG10}, SCC actions have \textit{purely exponential growth} (henceforth called a PEG action for brevity):
\begin{equation}\label{PEGdefn}
  |B(n)| \asymp \mathrm{e}^{n\omega_G}   
\end{equation}
where $\asymp$ denotes equality up to a multiplicative constant independent of  $n$. 
The assumption that the action is SCC in Theorem~\ref{MainThm} is essential and cannot be weakened to merely requiring purely exponential growth (PEG). Indeed, certain small cancellation quotients of geometrically finite actions on hyperbolic spaces exhibit no drop in growth rate (see \cite{YANG10}). These actions fail to satisfy the parabolic gap condition and therefore do not fall within the class of SCC actions.

To state the next corollary, we recall the following inequality established in \cite{CGYZ}, which generalizes earlier work of Coulon \cite{Coulon22} in the case of normal subgroups.
If  $G\act X$ has {purely exponential growth}, then for any nontrivial confined subgroup $H$ with  {non-degenerate}  confining subset, 
\begin{equation}\label{CoulonInequality}
\frac{\omega_{G/H}}{2} + \omega_{H}\ge \omega_G.   
\end{equation} 
   
As a consequence, we obtain the following corollary. 
\begin{cor}

Under the assumption of Theorem \ref{MainThm}, we have   $$\omega_{H}   \ge \omega_{G}-\omega_0/2>\omega_G/2.$$   
\end{cor}
This inequality $\omega_{H}>\omega_G/2$ was proved in \cite{CGYZ} via a completely different method using conformal measures on the boundary, and (\ref{CoulonInequality}) was established independently of that inequality. Thus, Theorem \ref{MainThm} offers an alternative proof.
Moreover, we emphasize that the  gaps $\omega_{G/H} < \omega_{G}$ and $\omega_{H} > \omega_{G}/2$, which are uniform over all confined subgroups $H$ sharing the same  confining set $P$, appear not to have been observed in previous works. 

\subsection{Applications to uniformly recurrent subgroups}\label{subsec:applications}
 Understanding the extent to which confined subgroups resemble, or differ from, normal subgroups is one of the main motivations for recent research. This question fits into the broader research theme of subgroup dynamics, concerning invariant random subgroups (IRS) introduced in \cite{AGV14} and uniformly recurrent subgroups (URS) in \cite{GW15}. 

 Let  $G$ be a locally compact, second countable topological group. The space $\mathrm{Sub}(G)$ of all closed subgroups in $G$, equipped with the Chabauty topology, is a compact metrizable space on which $G$ acts continuously by conjugation. When $G$ is discrete,  the Chabauty topology coincides with the product topology on $\{0,1\}^G$. Briefly, invariant random subgroups are random subgroups in $\mathrm{Sub}(G)$ whose distribution is a conjugation-invariant measure, and uniformly recurrent subgroups are topologically $G$-minimal systems in $\mathrm{Sub}(G)$ (i.e. every $G$-orbit is dense). Both IRSs and URSs serve as common generalizations of normal subgroups and lattices in semi-simple Lie groups. See \cite{ABBGNRS, FG23} for a more thorough discussion and relevant results in this direction. 
 
 We now return to the setting where $G$ is a discrete group and explain a connection between URSs and the confined subgroups studied in this paper.

 Let $\mathcal U\subset \mathrm{Sub}(G)$ be a \textit{nontrivial} $G$-minimal system; that is, $\mathcal{U}$ does not consist solely of the trivial subgroup  $\{1\}$, which is clearly a fixed point of $G$-conjugation action. Consequently, $\mathcal U$ is disjoint from some neighborhood of $\{1\}$, so there exists a finite set $P\subset G$ such that $H\cap (P\setminus \{1\})\ne \emptyset$ for every $H\in \mathcal U$. In particular, any URS $\mathcal U$ consists of confined subgroups with a common confining subset. Conversely, the closure of a confined subgroup in $\mathrm{Sub}(G)$ is disjoint from the trivial subgroup  (\cite[Proposition 10.2]{FG23}).


We derive the following corollary concerning the growth tightness of URSs.
\begin{cor}\label{URSCor}
Suppose $G$ admits  an SCC action on a proper geodesic metric space $X$ with contracting elements. Assume that $E(G)$ is trivial. Then every member of a uniformly  recurrent subgroup $\mathcal U$ is growth tight: there exists $\omega_0<\omega_G$ such that $\omega_{G/H}\le \omega_0$ for any $H\in \mathcal U$.     
\end{cor}

It is interesting to compare with the relevant results  in IRSs.
In \cite[Theorem 36]{AGV14}, examples of IRSs in free groups are constructed to demonstrate the failure of growth tightness. Thus, unlike URSs, there is no uniform gap property for the quotient growth of IRSs (which are a family of subgroups). 

\subsection{Negligible quotient growth for confined subgroups in PEG actions}
The well-known Hopf decomposition splits a measure-class-preserving action into its conservative and dissipative components. We do not require the precise definitions here; rather, we simply note that, roughly speaking, this decomposition serves as a measurable counterpart to the decomposition of the space into the limit set and domain of discontinuity. In \cite{Can20}, Cannizzo proved that sofic IRSs in free groups act conservatively on the boundary almost surely. As shown in \cite{CGYZ},  confined subgroups also act conservatively   on the horofunction boundary with respect to Patterson-Sullivan measures. It turns out that conservative actions are closely related to a notion weaker than growth tightness.
\begin{defn}
We say that a subgroup $H < G$ has \textit{negligible quotient growth} if
\[
\frac{|B_{G/H}(n)|}{|B_G(n)|} \longrightarrow 0 
\quad \text{as } n \to \infty.
\]
\end{defn}
In \cite{GKN}, Grigorchuk–Kaimanovich–Naghibeda proved that, for free groups, negligible quotient growth of a subgroup $H$ is equivalent to the conservativity of the induced action on the Gromov boundary with respect to the Patterson–Sullivan measure. This result was later generalized to hyperbolic groups in \cite{CGYZ}. 
We expect that the equivalence between negligible quotient growth and conservative boundary action should hold in substantially greater generality.

Our last result provides a supporting evidence in this direction by establishing negligible quotient growth for any confined subgroup, provided that $G\act X$ has {purely exponential growth}. 
  
    \begin{thm}\label{NegliGrowthThm}
		Suppose $G$ admits  a proper PEG action on a geodesic metric space $X$ with contracting elements. Let $H\subset G$ be a nontrivial confined subgroup with a finite non-degenerate confining subset. Then $H$ has negligible quotient growth.
	\end{thm} 
    As aforementioned,  the PEG actions in general do \textbf{not} imply growth tightness (\cite{YANG10}) and the negligible quotient growth is also weaker than growth tightness.  

    \subsection{Proof outline of    Theorem \ref{MainThm}}
    
    Let $[G/H]$ denote a complete set of right coset representatives of $H$ in $G$. 
Fix a basepoint $o \in X$ and write $\|g\| = d(o, go)$. Our goal is to prove that the growth rate of $[G/H] \subset G$ is strictly smaller than that of $G$. The first step allows us to restrict attention to the intersection of $[G/H]$ with a suitable class of linearly recurrent elements we now explain.

\subsubsection{\textbf{Genericity of linearly recurrent elements}}

We prove that a set of linearly recurrent elements is exponentially generic.

Let $\theta \in (0,1]$ and $M,L>0$. An element $g \in G$ is called \emph{$(\theta, L, M)$-linearly recurrent} if it admits a product decomposition
\[
g = s_1 \cdots s_m
\]
for some $m \ge \theta \|g\|$ such that the orbit points associated to the partial products 
\[
g_i := s_1 \cdots s_i \quad (1 \le i \le m)
\]
are $L$-separated and lie in $N_M([o, go])$. Roughly speaking, the segment $[o, go]$ spends at least a $\theta$-proportion of its length inside $N_M(Go)$.  See Definition~\ref{QconvexProperty}.

The key statement, proved in Lemma~\ref{OutGrowTight}, asserts that there exists a constant $M>0$ such that for all sufficiently large $L>0$ and $\theta = O(L^{-1})$, the set of $(\theta, L, M)$-linearly recurrent elements is exponentially generic in $G$.  

The dynamic nature of the terminology we adopt is motivated by the geodesic flow: when projected onto the quotient $X/G$, the geodesic segment $[o,go]$ is recurrent to the image of $N_M(Go)$ for a $\theta$-linear proportion of time. In this context, genericity   emerges as a consequence of the ergodic theorem when the geodesic flow is ergodic. In the geometric setup, the proof of the above result relies on the assumption that the action $G \act X$ is SCC (see Subsection~\ref{SSecSCC} for the precise definition).

\subsubsection{\textbf{Construction of the map}}

Let $g \in [G/H]$ be a $(\theta, L, M)$-linearly recurrent element with decomposition
$g = s_1 s_2 \cdots s_m$
as above. We insert suitable contracting elements into this product decomposition at positions specified by a binary vector 
\[
\epsilon = (\epsilon_0, \dots, \epsilon_{m-1}) \in \{0,1\}^m:
\]
whenever $\epsilon_i = 1$, we insert an appropriate element between $s_i$ and $s_{i+1}$.

More precisely, define a map
\[
\Phi_g : \{0,1\}^m \to Hg
\]
by
\[
\Phi_g(\epsilon_0, \dots, \epsilon_{m-1})
= \prod_{i=0}^{m-1} (f_i p_i f_i^{-1})^{\epsilon_i} s_{i+1}.
\]
Here $f_i$ are chosen from a fixed finite set $F$ of contracting elements, and $p_i \in P$ are selected from the \emph{confining} subset $P$ so that
\[
g_i f_i p_i f_i^{-1} g_i^{-1} \in H.
\]
Here $F$ depends only on $P$ rather than $H$: this is the point to get the gap independent of $H$ in Theorem~\ref{MainThm}.
This choice ensures that $\Phi_g(\epsilon) \in Hg$ (see Lemma~\ref{MaptoCoset}). Following the approach of \cite{CGYZ} (reviewed in \textsection~\ref{SSecCGYZ}), we prove that the resulting word labels a quasi-geodesic (see Lemma~\ref{AugPathAdm}). Furthermore, for any fixed $\theta$, provided $L$ is sufficiently large, the map $\Phi_g$ is injective (see Lemma~\ref{PhiInj}).

\subsubsection{\textbf{Completion of the proof}}

Since $\Phi_g(\{0,1\}^m) \subset Hg$, the images corresponding to distinct $(\theta, L, M)$-linearly recurrent elements $g \in [G/H]$ are disjoint in $G$. As $m \ge \theta \|g\|$, the injectivity of $\Phi_g$ yields $2^m$ distinct elements in $Hg$ for each such $g$.

The resulting exponential amplification by a factor of $2^{\theta \|g\|}$ shows that the growth rate of $[G/H]$ is strictly smaller than $\omega_G$. See Lemma~\ref{CritGapLem} for the abstract growth gap criterion. This completes the proof of growth tightness in Theorem~\ref{MainThm}.
\\
\paragraph{\textbf{Related works.}}
To conclude the Introduction, let us compare the above strategy with the classical approach originating in \cite{Sam2,DPPS} (further developed in \cite{ACTao,YANG10}). The arguments in \cite{Sam2,DPPS} rely crucially on the assumption that $H$ is a normal subgroup. In that setting, one key consequence is that $[G/H]$ is \textit{contained} in the set of elements of $G$ that avoid a fixed element $g$ as a subword (see \cite[Corollary 4.6]{YANG10}). The growth tightness of the latter set   proved  in \cite{YANG10} completes the proof. 

This approach unfortunately breaks down when $H$ is merely a general confined subgroup, as the corresponding subword-avoidance property no longer holds. Furthermore, the growth gap criterion employed here in Lemma~\ref{CritGapLem} differs substantially from that of \cite[Lemma 2.23]{YANG10}.  The  strategy presented here was developed following the solution by Lihuang Ding and Kairui Liu to the case of Theorem~\ref{MainThm} for confined subgroups of free groups; their solution is included in the appendix.   

    
     \subsubsection*{\textbf{Organization of the paper}} Section~\ref{SecPrelim} contains the necessary preliminaries: we recall the definition of contracting geodesics and statistically convex-cocompact actions, prove in Lemma~\ref{OutGrowTight} that generic elements are linearly recurrent, and establish a gap criterion for growth rates (Lemma~\ref{CritGapLem}). In Section~\ref{SecNegliGrowth}, we give a variant of the extension lemma adapted to confined subgroups (Lemma~\ref{ExtLemConf}) and prove that the quotient of a confined subgroup has negligible growth (Theorem~\ref{NegligibleGrowthThm}). Section~\ref{SecGrowthTight} is then dedicated to the growth tightness of linearly recurrent elements modulo a confined subgroup (Theorem~\ref{AGroTig}), from which Theorem~\ref{MainThm} follows. The paper concludes with an appendix by Lihuang Ding and Kairui Liu, which presents two alternate proofs of Theorem~\ref{MainThm} for confined subgroups in free groups.
     \ack
     The second-named author thanks  Inhyeok Choi, Ilya Gekhtman and Tianyi Zheng for helpful discussions during the project  \cite{CGYZ}. Theorem~\ref{MainThm} was originally posed with a negative answer anticipated (as in random invariant subgroups). This work began after (surprising) proofs  in free groups of Theorem \ref{MainThm} were found independently by the first-named author and Kairui Liu.

	\section{Preliminaries}\label{SecPrelim}
    \subsection{Notation and convention}\label{subsec:notation}
    Denote by $\str$ the set of binary strings of length $m$, i.e.
\[
\bar\epsilon\in\str \quad\Longleftrightarrow\quad 
\bar\epsilon=(\epsilon_0,\epsilon_1,\dots,\epsilon_{m-1})\ \text{ with }\ \epsilon_i\in\{0,1\}.
\]
Set
\[
\|\bar\epsilon\|:=\sum_{i=1}^m \epsilon_i .
\]
Observe that $\str$ can be identified with the power set $P(m)$ of $\{0,1,2,\dots,m-1\}$ via the bijection
\[
\bar\epsilon \longmapsto I(\bar\epsilon):=\{\,i:\ 0\le i\le m-1,\ \epsilon_i=1\,\}\in P(m).
\]
In particular, $|I(\bar\epsilon)|=\|\bar\epsilon\|$.

Let $(\U, d)$ be a proper geodesic metric space.     
Let $\alpha :[s,t]\subset \mathbb R\to\U$ be a path parametrized by arc-length from the initial point $\alpha^-:=\alpha(s)$ to terminal point $\alpha^+:=\alpha(t)$. 
    Given two points $x,y$ lying on the image of $\alpha$, denote by $[x,y]_\alpha$ the parametrized subpath of $\alpha$ from $x$ to $y$.
For two points $x,y\in \U$, we write $[x, y]$ for a choice of geodesic between $x$ and $y$. 
 
A path $\alpha$ is called a \textit{$c$-quasi-geodesic} for $c\ge 1$ if for   any rectifiable subpath $\beta$,
$$\len(\beta)\le c \cdot d(\beta^-, \beta^+)+c$$
 where   $\len(\beta)$ denotes the length of $\beta$. 
 
Denote by $\alpha\cdot \beta$ (or simply $\alpha\beta$) the concatenation of two paths $\alpha, \beta$  provided that $\alpha^+ = \beta^-$.

    We frequently construct a path labeled by a word $(g_1, g_2,\cdots,g_n)$, which by convention means the following concatenation
$$
[o,g_1o]\cdot g_1[o,g_2o]\cdots (g_1\cdots g_{n-1})[o,g_no]
$$
where the basepoint $o$ is understood from context. With this convention, the paths labeled by $(g_1,g_2,g_3)$ and $(g_1g_2, g_3)$ differ in general, depending on whether  $[o,g_1o]\cdot g_1[o,g_2o]$ is a geodesic. 

	\subsection{Contracting geodesics}
Let $Z$ be a closed subset of $\U$ and $x$ be a point in $\U$.  Denote by $d(x, Z)$ the distance between $x$ and $Z$, \emph{i.e.} 
\[
d(x, Z) : = \inf \big \{ d(x, y): y \in Z \big \}. 
\]
Let
\[ \pi_{Z}(x) : = \big \{ y\in Z: d(x, y) = d(x, Z) \big \} \]
be the set of closet point projections from $x$ to $Z$, and let $N_r(Z) : = \{y\in \U: d(y, Z) \le r\}$ be the closed $r$-neighborhood of $Z$. Since $X$ is proper, 
$\pi_{Z}(x)$ is nonempty. We refer to $\pi_{Z}(x) $ as the \emph{projection set} of $x$ to $Z$. Define $\proj_Z(x,y):=\diam{\pi_Z(x)\cup \pi_Z(y)}$. Then $\proj_Z$ satisfies the triangle inequality: for any closed subset $Z\subset \U$ and any $x, y, z\in \U$, $$\proj_Z(x,z) \le \proj_Z(x,y) + \proj_Z(y, z).$$

\begin{defn} \label{Def:Contracting}
We say that a closed subset $Z \subseteq X$ is \emph{$C$-contracting} for a constant $C>0$ if,
for all pairs of points $x, y \in \U$, we have
\[
d(x, y) \leq d(x, Z) \quad  \Longrightarrow  \quad \proj_Z(x,y) \leq  C.
\]
Such a $C$ is called a \emph{contracting constant} for $Z$. A collection of $C$-contracting subsets is referred to as a $C$-contracting system.

An element $h\in \isom(\U)$ is called \textit{contracting} if it acts by translation on a contracting bi-infinite quasi-geodesic. Equivalently, the map $n\in \mathbb Z\longmapsto h^no$ is a quasi-geodesic whose image is contracting.  
\end{defn}


The contracting property admits several equivalent characterizations (see e.g. \cite[Lemma 3.2]{BF1},\cite[Lemma 2.4]{Sisto}). When we speak of a $C$-contracting property, the constant $C$ is assumed to satisfy  the following three statements.

\begin{lem}\label{BigThree}
Let $U$ be a contracting subset. Then there exists $C>0$ such that 
\begin{enumerate}
\item
If $d(\gamma,
U) \ge C$ for a geodesic  $\gamma$, then 
$\diam{\pi_U (\gamma)}  \le C$.
\item
If $\diam{\pi_U (\gamma)}  \ge  C$, then $d(\pi_U(\gamma^-),\gamma)\le C$ and $d(\pi_U(\gamma^+),\gamma)\le C$.
\item
For any metric ball $B$ disjoint from $U$, we have $\diam{\pi_U(B)}\le C$.
\end{enumerate}
\end{lem}

Contracting subsets are Morse, and the property is preserved up to finite Hausdorff distance. The following properties will be used later. Their proofs follow easily from Lemma \ref{BigThree} and left to the interested reader.
\begin{lem}\label{BigFive}
Let $U\subseteq \U$ be a  $C$-contracting subset for $C>0$. Then the following hold.

\begin{enumerate}


\item
For any  geodesic $\gamma$, we have $$\big |\proj_U(\gamma^-,\gamma^+)- \diam{\pi_U(\gamma)}\big|\leq 4C.$$

\item
For any $y,z\in \U$, $\proj_U(y, z)\le d(y, z)+ 2C$ .

\end{enumerate}
\end{lem}

A group is called \textit{elementary} if it is virtually $\mathbb Z$ or finite. 
Let $h\in G$ be a contracting element. Define the coarse orbit stabilizer  $$E(h):=\{g\in G: d_H(\langle h\rangle o, g \langle h\rangle o)<\infty\},$$
where $d_H$ denotes the Hausdorff distance.
Any elementary subgroup containing $h$ is clearly contained in $E(h)$. In fact, $E(h)$ is the maximal such subgroup and admits the following algebraic description.
\begin{lem}\cite[Lemma 2.11]{YANG10}\label{elementarygroup}
For a contracting element $h\in G$,  $E(h) $ is  the maximal elementary subgroup containing $h$ and satisfies
$$
E(h)=\{g\in G: \exists n\in \mathbb N_{> 0}, (\;gh^ng^{-1}=h^n)\; \lor\;  (gh^ng^{-1}=h^{-n})\}
$$
\end{lem} 

Fixing the basepoint $o\in\U$, the \textit{axis} of $h$  is defined as the following quasi-geodesic: 
\begin{equation}\label{axisdefn}
\ax(h)=\{f o: f\in E(h)\}.
\end{equation}
Note that $\ax(h)=\ax(k)$ and $E(h)=E(k)$    for any contracting element   $k\in E(h)$. Given two contracting elements $g, h\in G$, we say that they are \emph{$B$-independent} (or simply independent if such a $B$ exists ) if there exists $B>0$ such that $$\proj_U(V)\le B$$ for any $U = p\ax(g)$ and $V = q\ax(h)$. A set of contracting elements is said to consist of independent elements if they are pairwise independent.

\subsection{Extension Lemma}

{We fix a finite set $F \subseteq G$ of independent $C$-contracting elements and let $\f = \{g \ax(f) : f \in F, g \in G\}$ be the collection of all left translates of their axes.} The following notion of an admissible path provices a way to construct   a quasi-geodesic  by concatenating geodesics via elements of $\f$.
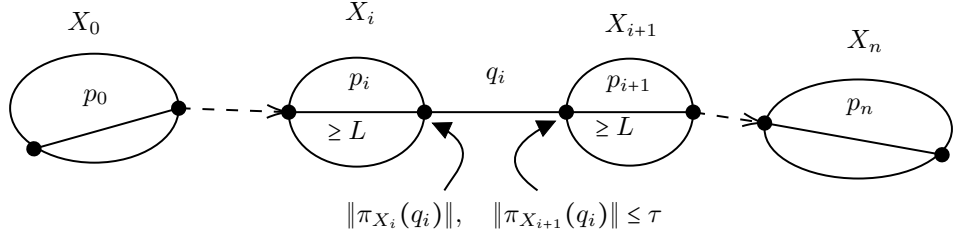
\begin{figure}
    \centering

\tikzset{every picture/.style={line width=0.75pt}} 

\begin{tikzpicture}[x=0.75pt,y=0.75pt,yscale=-1,xscale=1]

\draw   (80.5,98.5) .. controls (80.5,83.31) and (99.53,71) .. (123,71) .. controls (146.47,71) and (165.5,83.31) .. (165.5,98.5) .. controls (165.5,113.69) and (146.47,126) .. (123,126) .. controls (99.53,126) and (80.5,113.69) .. (80.5,98.5) -- cycle ;
\draw   (221,100.5) .. controls (221,85.31) and (236.33,73) .. (255.25,73) .. controls (274.17,73) and (289.5,85.31) .. (289.5,100.5) .. controls (289.5,115.69) and (274.17,128) .. (255.25,128) .. controls (236.33,128) and (221,115.69) .. (221,100.5) -- cycle ;
\draw   (461,109) .. controls (461,95.19) and (482.15,84) .. (508.25,84) .. controls (534.35,84) and (555.5,95.19) .. (555.5,109) .. controls (555.5,122.81) and (534.35,134) .. (508.25,134) .. controls (482.15,134) and (461,122.81) .. (461,109) -- cycle ;
\draw   (361,100.5) .. controls (361,85.59) and (375.33,73.5) .. (393,73.5) .. controls (410.67,73.5) and (425,85.59) .. (425,100.5) .. controls (425,115.41) and (410.67,127.5) .. (393,127.5) .. controls (375.33,127.5) and (361,115.41) .. (361,100.5) -- cycle ;
\draw  [dash pattern={on 4.5pt off 4.5pt}]  (165.5,98.5) -- (219,100.43) ;
\draw [shift={(221,100.5)}, rotate = 182.06] [color={rgb, 255:red, 0; green, 0; blue, 0 }  ][line width=0.75]    (10.93,-3.29) .. controls (6.95,-1.4) and (3.31,-0.3) .. (0,0) .. controls (3.31,0.3) and (6.95,1.4) .. (10.93,3.29)   ;
\draw  [dash pattern={on 4.5pt off 4.5pt}]  (425,100.5) -- (459.02,105.7) ;
\draw [shift={(461,106)}, rotate = 188.69] [color={rgb, 255:red, 0; green, 0; blue, 0 }  ][line width=0.75]    (10.93,-3.29) .. controls (6.95,-1.4) and (3.31,-0.3) .. (0,0) .. controls (3.31,0.3) and (6.95,1.4) .. (10.93,3.29)   ;
\draw    (289.5,100.5) -- (361,100.5) ;
\draw [shift={(361,100.5)}, rotate = 0] [color={rgb, 255:red, 0; green, 0; blue, 0 }  ][fill={rgb, 255:red, 0; green, 0; blue, 0 }  ][line width=0.75]      (0, 0) circle [x radius= 3.35, y radius= 3.35]   ;
\draw [shift={(289.5,100.5)}, rotate = 0] [color={rgb, 255:red, 0; green, 0; blue, 0 }  ][fill={rgb, 255:red, 0; green, 0; blue, 0 }  ][line width=0.75]      (0, 0) circle [x radius= 3.35, y radius= 3.35]   ;
\draw    (299.5,140) .. controls (305.29,130.35) and (319,118.84) .. (296.13,107.26) ;
\draw [shift={(293.5,106)}, rotate = 24.36] [fill={rgb, 255:red, 0; green, 0; blue, 0 }  ][line width=0.08]  [draw opacity=0] (8.93,-4.29) -- (0,0) -- (8.93,4.29) -- cycle    ;
\draw    (346.5,140) .. controls (338.7,134.15) and (322.34,121.65) .. (353.04,106.19) ;
\draw [shift={(355.5,105)}, rotate = 154.8] [fill={rgb, 255:red, 0; green, 0; blue, 0 }  ][line width=0.08]  [draw opacity=0] (8.93,-4.29) -- (0,0) -- (8.93,4.29) -- cycle    ;
\draw    (92.5,119) -- (165.5,98.5) ;
\draw [shift={(165.5,98.5)}, rotate = 344.31] [color={rgb, 255:red, 0; green, 0; blue, 0 }  ][fill={rgb, 255:red, 0; green, 0; blue, 0 }  ][line width=0.75]      (0, 0) circle [x radius= 3.35, y radius= 3.35]   ;
\draw [shift={(92.5,119)}, rotate = 344.31] [color={rgb, 255:red, 0; green, 0; blue, 0 }  ][fill={rgb, 255:red, 0; green, 0; blue, 0 }  ][line width=0.75]      (0, 0) circle [x radius= 3.35, y radius= 3.35]   ;
\draw    (221,100.5) -- (289.5,100.5) ;
\draw [shift={(221,100.5)}, rotate = 0] [color={rgb, 255:red, 0; green, 0; blue, 0 }  ][fill={rgb, 255:red, 0; green, 0; blue, 0 }  ][line width=0.75]      (0, 0) circle [x radius= 3.35, y radius= 3.35]   ;
\draw    (361,100.5) -- (425,100.5) ;
\draw [shift={(425,100.5)}, rotate = 0] [color={rgb, 255:red, 0; green, 0; blue, 0 }  ][fill={rgb, 255:red, 0; green, 0; blue, 0 }  ][line width=0.75]      (0, 0) circle [x radius= 3.35, y radius= 3.35]   ;
\draw    (461,106) -- (550.5,122) ;
\draw [shift={(550.5,122)}, rotate = 10.14] [color={rgb, 255:red, 0; green, 0; blue, 0 }  ][fill={rgb, 255:red, 0; green, 0; blue, 0 }  ][line width=0.75]      (0, 0) circle [x radius= 3.35, y radius= 3.35]   ;
\draw [shift={(461,106)}, rotate = 10.14] [color={rgb, 255:red, 0; green, 0; blue, 0 }  ][fill={rgb, 255:red, 0; green, 0; blue, 0 }  ][line width=0.75]      (0, 0) circle [x radius= 3.35, y radius= 3.35]   ;

\draw (108,48.4) node [anchor=north west][inner sep=0.75pt]    {$X_{0}$};
\draw (248,43.4) node [anchor=north west][inner sep=0.75pt]    {$X_{i}$};
\draw (379,49.4) node [anchor=north west][inner sep=0.75pt]    {$X_{i+1}$};
\draw (501,57.4) node [anchor=north west][inner sep=0.75pt]    {$X_{n}$};
\draw (200,144.4) node [anchor=north west][inner sep=0.75pt]    {$\diam{ \pi _{X_{i}}( q_{i}) } \leq \tau,$};
\draw (322,144.4) node [anchor=north west][inner sep=0.75pt]    {$\diam{ \pi _{X_{i+1}}( q_{i}) } \leq \tau $};
\draw (250,79.4) node [anchor=north west][inner sep=0.75pt]    {$p_{i}$};
\draw (116,87.4) node [anchor=north west][inner sep=0.75pt]    {$p_{0}$};
\draw (380,80.4) node [anchor=north west][inner sep=0.75pt]    {$p_{i+1}$};
\draw (319,76.4) node [anchor=north west][inner sep=0.75pt]    {$q_{i}$};
\draw (501,92.4) node [anchor=north west][inner sep=0.75pt]    {$p_{n}$};
\draw (239,103.4) node [anchor=north west][inner sep=0.75pt]    {$\geq L$};
\draw (374,101.4) node [anchor=north west][inner sep=0.75pt]    {$\geq L$};

\end{tikzpicture}
    \caption{Admissible path}
    \label{fig:admissiblepath}
\end{figure}
\begin{defn}[Admissible Path]\label{AdmDef}
Given $L,\tau\geq0$, a path $\gamma$ is called $(L,\tau)$-\textit{admissible} in $\U$ if $\gamma$ is a concatenation of geodesics $p_0q_1p_1\cdots q_np_n$ $(n\in\mathbb{N})$ such that the two endpoints of each $p_i$ lie in some $X_i\in \f$, and   the following   \textit{Long Local} and \textit{Bounded Projection} properties hold:
\begin{enumerate}
\item[(LL)] Each $p_i$  for $1\le i< n$ has length greater than $L$, while $p_0$ and $p_n$ could be trivial;
\item[(BP)] For each $X_i$, we have $X_i\ne X_{i+1}$ and $$\max\{\diam{\pi_{X_i}(q_i)},\diam{\pi_{X_i}(q_{i+1})}\}\leq\tau,$$ where by convention $q_0:=\gamma^-$ and $q_{n+1}:=\gamma^+$.
\end{enumerate} 
The collection $\{X_i: 1\le i\le n\}$ is referred to as the contracting subsets associated with the admissible path.
\end{defn}

\begin{rem}\label{ConcatenationAdmPath}
    The geodesics $q_i$ could are allowed to be trivial. In that case condition (BP) reduces to checking $X_i\ne X_{i+1}$. Admissible paths can be concatenated as follows: let $p_0q_1p_1\cdots q_np_n$ and $p_0'q_1'p_1'\cdots q_n'p_n'$ be $(L,\tau)$-admissible. If $p_n=p_0'$ has length greater than $L$, then the concatenation $(p_0q_1p_1\cdots q_np_n)\cdot (q_1'p_1'\cdots q_n'p_n')$ naturally inherits an $(L,\tau)$-admissible structure.  
\end{rem}

    The following lemma is frequently used to verify $X_i\neq X_{i-1}$ in property (BP):
    \begin{lem}\label{LL2ImplyBP}
        Using notations in the definition \ref{AdmDef} of admissible path, for all $1\le i\le n$, if the following hold $$\max\{ \diam{\pi_{X_{i}}(q_{i})}, \diam{\pi_{X_{i-1}}(q_{i})}\}\leq\tau \quad \text{and} \quad \ell(q_{i}) > \tau,$$ then $X_i \neq X_{i-1}$.
    \end{lem}
    \begin{proof}
        Suppose to the contrary that $X_i = X_{i-1}$. Let $a = (q_i)^-$ and $b = (q_i)^+$, so that $a = (p_{i-1})^+$, $b = (p_i)^-$. Since $p_i\subseteq X_i$ and $p_{i-1}\subseteq X_{i-1}$, we have $a, b\in X_i$. But $d(a, b) = \proj_{X_i}(a, b) \le \diam{\pi_{X_i}(q_i)} \le \tau$, contradicting $\ell(q_i) > \tau$.
    \end{proof}

A sequence of points $x_i$ on a path $p$   is called \textit{linearly ordered} if $x_{i+1}\in [x_i, p^+]_p$ for each $i$.

\begin{defn}[Fellow travel]\label{Fellow}
Let   $\gamma = p_0 q_1 p_1 \cdots q_n p_n$ be an $(L, \tau)-$admissible
path. We say $\gamma$ has \textit{$r$-fellow travel} property for some $r>0$   if for any geodesic  
$\alpha$  with the same endpoints as $\gamma$,   there exists a sequence of linearly ordered points $z_i, w_i$ ($0 \le i \le n$) on $\alpha$ such that  
$$d(z_i, p_{i}^-) \le r,\quad d(w_i, p_{i}^+) \le r.$$
In particular, the intersection $N_r(X_i)\cap \alpha$ has diameter at least $\ge L-2r$ for each $X_i\in \f$ associated with $\gamma$. 
\end{defn}
The following result ensures that a locally long admissible path enjoys the fellow travel property.

\begin{prop}\label{admisProp}\cite[Prop 3.1]{YANG6}
For any $\tau>0$, there exist $L,  r, c>0$ depending only on $\tau$ and $C$ such that every $(L, \tau)$-admissible path   has $r$-fellow travel property. Consequently, it is a $c$-quasi-geodesic.
\end{prop}

The next lemma provides a method for constructing admissible paths.
\begin{lem}[Extension Lemma]\label{extend3}

There exist constants  $L, r, B>0$ depending only on $C$ with the following property. 

Let $h_{1}, h_{2}, h_{3} \in G$ be $B$-independent contracting elements. 
For each $1\le i\le 3$, choose an element $f_i\in \langle h_i\rangle$ such that $\|Fo\|_{\min}\ge L$, and let $F$ be the set of these elements. Then for any $g,h\in G$, there exists $f \in F$ such that the path  $$\gamma:=[o, go]\cdot(g[o, fo])\cdot(gf[o,ho])$$ is $(L, \tau)$-admissible with respect to the contracting system $\f$. 
\end{lem}

\begin{rem} 
Since admissibility is a local conditions, one can connect any number of elements  $g\in G$ using $F$ to satisfy (1) and (2), and to obtain a admissible path. We refer the reader to \cite{YANG10} for a precise formulation.
\end{rem}
The following Corollary will be useful in Section~\ref{SecGrowthTight}.
\begin{cor}[Corollary 3.9 in \cite{YANG6}]\label{ProjMidSet}
    Let $\gamma$ be an $(L, \tau)$-admissible path where $L$ and $\tau$ are given by Proposition~\ref{admisProp}, and let $X_k$ be one of its associated contracting subset $(0\le k\le n)$. Then there exists $N = N(C, \tau) > 0$ such that $$\proj_{X_k}(\gamma_{-}, (p_k)_{-}) < N, \quad \proj_{X_k}(\gamma_+, (p_k)_+)<N$$ where $\gamma_{-}$,$\gamma_+$ and $(p_k)_{-}$, $(p_k)_+$ denote the starting and ending points of $\gamma$ and $p_k$, respectively.
\end{cor}

\subsection{Statistically convex-cocompact actions}\label{SSecSCC}
We now introduce the class of statistically convex-cocompact actions, introduced in \cite{YANG10} as a generalization of convex-cocompact actions. This notion independently appears under the name \textit{strongly positively recurrent} manifold in the dynamical context of  Schapira-Tapie \cite{ST21}.

Given constants $0\leq M_1\leq M_2$, let $\mathcal{O}_{M_1,M_2}$ denote the set of elements $g\in G$ for which there exists a geodesic $\gamma$ joining a point in $N_{M_2}(o)$ to a point in $N_{M_2}(go)$ whose interior lies outside $N_{M_1}(G o)$.

\begin{defn}[SCC Action]\label{SCCDefn}
If there exist positive constants $M_1,M_2>0$ such that $\e {\mathcal{O}_{M_1,M_2}}<\e G<\infty$, then the proper action of $G$ on $\U$ is called \textit{statistically convex-cocompact} (SCC).
\end{defn}
\begin{rem}
The definition of $\mathcal{O}_{M_1,M_2}$ is motivated by the action of the fundamental group of a finite-volume negatively curved Hadamard manifold on its universal cover. In that situation, for appropriate constants $M_1, M_2$, the set $\mathcal{O}_{M_1,M_2}$ coincides with  the union of the orbits of cusp subgroups up to finite Hausdorff distance. The condition $\e {\mathcal{O}_{M_1,M_2}}<\e G$ is known as the \textit{parabolic gap condition} in the work of Dal'bo, Otal and Peign\'{e} cite{DOP}. The growth rate $\e {\mathcal{O}_{M_1,M_2}}$ is called \textit{complementary growth exponent} in \cite{ACTao} and \textit{entropy at infinity} in \cite{ST21}. 
\end{rem}

SCC actions have purely exponential growth, and are therefore of divergence type.  
\begin{lem}\cite[Theorem B(3)]{YANG10}
Suppose that $G\act \U$ is a non-elementary SCC action with contracting elements. Then $G$ has purely exponential growth: for all $n\ge 1$, $$|B(n)|\asymp \mathrm{e}^{\e G n}.$$    
\end{lem}

\begin{defn}\label{barriers}
Fix $\epsilon>0$ and a subset $P\subseteq G$.  A geodesic $\gamma$ is said to contain an \textit{$(\epsilon, f)$-barrier} for some $f\in P$   if there exists    an element $g \in G$ such that 
\begin{equation}\label{barrierEQ}
\max\{d(g\cdot o, \gamma), \; d(g\cdot fo, \gamma)\}\le \epsilon.
\end{equation}
If $\gamma$ contains no {$(\epsilon, f)$-barrier} for any $f\in P$, then it is called \textit{$(\epsilon, P)$-barrier-free}.
\end{defn}
The terminology ``barrier" reflects the obstruction these segments present to the geometry of the geodesic. Inequality \eqref{barrierEQ} implies that a translate of the axis of $f$ fellow-travels with $\gamma$, indicating the presence of a 'long' axis. Since our goal is to enumerate elements characterized by short axes, the appearance of such a fellow-traveling segment effectively acts as an obstruction.

\begin{defn}\label{FractionalBarrierFree}
Fix $\theta\in (0,1]$, $\epsilon, M, L>0$ and a subset $P\subseteq G$. An element $g\in G$ is said to be $(\theta, L)$\emph{-fractional }$(\epsilon, P)$\emph{-barrier-free} if there exists a collection $\mathbb K$ of disjoint open subintervals  $\alpha$ of the geodesic $[o, go]$ satisfying:
\begin{itemize}
    \item each $\alpha\in \mathbb K$ is $(\epsilon, P)$-barrier-free, has length at least $L$, and its endpoints $\partial \alpha$ are contained in $N_M(Go)$;
    \item the total length of these intervals satisfies
    $$
    \sum_{\alpha\in \mathbb K} \len(\alpha)\ge \theta \; d(o, go).
    $$
\end{itemize}

Denote by  $\mathcal V_{\epsilon, M, P}({\theta, L})$  the set of elements $g\in G$ possessing the   $(\theta, L)$-fractional $(\epsilon, P)$-barrier-free property.
\end{defn}

\begin{thm}\label{FGrowthTightThm}\cite[Theorem 5.2]{GYANG}
There exists constants $\epsilon, M>0$ such that the following holds:
For any $0<\theta\le 1$ and  any nontrivial  subset $P\subseteq G$, there exists $L=L(\theta, P)>0$ such that $\mathcal V_{\epsilon, M, P}({\theta, L})$  is a growth tight set.
\end{thm}

Let $\epsilon, M>0$ be the constants provided by Theorem \ref{FGrowthTightThm}. For $g\in G$, let $\mathbb{K}$ be the collection of  connected components $\alpha$ of length at least $L$ in $[o, go]\setminus N_M(Go)$. Denote by $\mathcal{O}_M({\theta, L})$ the set of $g\in G$ such that $$\sum_{\alpha\in \mathbb{K}} \ell(\alpha) \ge \theta \; d(o, go).$$
By \cite[Lemma 6.1]{YANG10}, each component $\alpha$ are $(\epsilon, P)$-barrier-free . The following corollary will play an important role  in what follows.
\begin{cor}\label{OGrowthTight}\cite[Corollary 5.3]{GYANG}
    For any $\theta\in (0, 1]$ there exists $L = L(\theta)$ such that $\mathcal{O}_M({\theta, L})$ is a growth tight set.
\end{cor}

    By definition of $\mathcal{O}_M(\theta, L)$, if $L_1 > L_2$ and $M_1 > M_2$, then any connected component $\alpha$ of length at least $L_1$ also has length at least $L_2$, and $N_{M_2}(Go)\subseteq N_{M_1}(Go)$. Consequently, $\mathcal{O}_{M_1}({\theta, L_1}) \subseteq \mathcal{O}_{M_2}({\theta, L_2})$.

    Before proceeding, let us introduce some additional terminology.
    A \textit{product decomposition} of $g\in G$ means a sequence $(s_1, \cdots, s_m)$ where $s_i \in G$ and $s_i\neq 1$ for $1\le i\le m$ such that $g =  s_1 \cdots s_m$. Set $g_0=1$ and $g_i=s_1\cdots s_i$ for $1\le i\le m$.

    \begin{defn}
    Fix a point $o\in X$.  A product decomposition $g = s_1 \cdots s_m$ is called \textit{$M$-almost geodesic} if for all $0\le i < j < k \le m$, 
    $$d(g_i o, g_jo) + d(g_jo, g_k o) \le d(g_io, g_k o) + M.$$
    Equivalently, the sequence $\{g_io:0\le i\le m\}$ forms a $(1,M)$-quasi-geodesic.
    \end{defn}

    We now introduce a class of elements admitting a specific type of almost geodesic decomposition. 
    Let $\gamma = [a, b]$ be a geodesic. Recall that a set of points $\{x_i\in \gamma : 0\le i\le m+1\}$ is   linearly ordered on $\gamma$ if   $d(a, x_i)\le d(a, x_{i+1})$ for each $0\le i \le m$.

    \begin{defn}\label{QconvexProperty}
    
     Let $\theta\in (0,1]$ and $L, M>0$. An element $g\in G$ is said to be \emph{$(\theta, L, M)$-linearly recurrent} if there exists a geodesic  $\gamma=[o,go]$ and a linearly ordered set of distinct points  $$\{x_i: 0\le i\le m\}$$ on $\gamma$ with $x_0=o$, $x_m=go$ and $m:=\lfloor \theta d(o,go)\rfloor$, together with a sequence of elements in $G$ $$\{g_i: 1\le i\le m\}$$ where $g_0 = 1$ and $g_{m} = g$, such that  $$d(g_i o, x_i)\le M \quad \text{for all}\; 1\le i\le m$$ and  $$d(g_i o, g_j o)\ge L  \quad \text{for all}\;  0\le i \ne j\le m.$$   

     Setting $s_i = g_{i-1}^{-1} g_{i}$ for $1\le i\le m$ yields a product decomposition $g = s_1 \cdots s_m$.

    \end{defn}
    \begin{lem}\label{lem:QuasiConvex->AlmostGeodesic}
        If $g\in G$ is  $(\theta, L, M)$-linearly recurrent, then the product decomposition $g = s_1 \cdots s_m$ is $6M$-almost geodesic.
    \end{lem}
    \begin{proof}
        Let $g_i\in G$ and $x_i\in [o,go]$ be as in Definition~\ref{QconvexProperty}.
        Fix any $0\le i<j<k\le m$. By the triangle inequality, 
        \[d(g_i o, g_j o)
        \le d(g_i o, x_i) + d(x_i,x_j) + d(x_j,g_jo)
        \le d(x_i,x_j) + 2M.\]
        and similarly, 
        \[d(g_j o, g_k o)
        \le d(g_j o, x_j) + d(x_j,x_k) + d(x_k,g_ko)
        \le d(x_j,x_k) + 2M.\]
        Since $x_i,x_j,x_k$ are linearly ordered on the geodesic $[o, go]$, we have $$d(x_i,x_j) + d(x_j,x_k) = d(x_i,x_k).$$
        Applying the triangle inequality again, 
        \[
        d(x_i,x_k)
        \le d(x_i,g_io) + d(g_io,g_ko) + d(g_ko,x_k)
        \le d(g_io,g_ko) + 2M.
        \]
        Combine the inequalities yields
        \[\begin{aligned}
        d(g_io,g_jo)+d(g_jo,g_ko)
        &\le d(x_i,x_j)+d(x_j,x_k)+4M\\
        &=d(x_i,x_k)+4M\\
        &\le d(g_io,g_ko) + 6M.
        \end{aligned}\]
        Since $i,j,k$ are arbitrary, the product decomposition is $6M$-almost geodesic.
    \end{proof}
    
    Let $\mathcal{D}_M(\theta, L)$ denote the set of $(\theta, L, M)$-linearly recurrent elements in $G$.
    Observe that if $\theta_1 \ge \theta_2$, $L_1 \ge L_2$, and $M_1 \le M_2$, then $\mathcal{D}_{M_1}(\theta_1, L_1) \subseteq \mathcal{D}_{M_2}(\theta_2, L_2)$.
    \begin{lem}\label{OutGrowTight}
    Assume that the action $G\act X$ is SCC. Then there exists $M > 0$ with the following property: for any $L>0$, there exists $\theta>0$ such that the set $G\setminus \mathcal{D}_M(\theta, L)$ is growth tight.    
    \end{lem}
    \begin{proof}
        By Corollary \ref{OGrowthTight}, for any $0<\theta_0\le 1/8$, there exist constants $M_0, L_0>0$ such that $\mathcal{O}_{M_0}(\theta_0, L_0)$ is growth tight. For example, we take $\theta_0 = 1/8$. Let $$M = \max\{2 M_0, L_0\}, \quad \theta = {\theta_0}/{(2M + L)}$$

        Our goal is to prove that the union  $\mathcal{O}_{M_0}({\theta_0, L_0}) \cup \mathcal{D}_M(\theta, L)$ covers the whole group $G$ up to finitely many exceptions. This clearly concludes the proof since $\mathcal{O}_{M_0}(\theta_0, L_0)$ is growth tight.

        Let $N = 16(4M + L)$. Let   $g\in G$ with $\|g\| = d(o, go)>N$.  We shall prove that $g$ lies in  $\mathcal{O}_{M_0}({\theta_0, L_0}) \cup \mathcal{D}_M(\theta, L)$ via the two claims given below. 
        
        First, let us explain the related construction by fixing a geodesic $\gamma = [o, go]$. Choose a linearly ordered set of points $\{x_i: 1\le i\le k\}$ on $\gamma$ with $$k = \left\lfloor \frac{\|g\|}{2M}\right\rfloor - 1$$ such that $d(x_i, x_{j})\ge 2M$ for $0\le i < j \le k+1$, where we set $x_0 = o$, $x_{k+1} = go$. 
        Since $\|g\| = d(o, g o) > N$, we have $$Mk > \frac{\|g\|}{2} - 2M > \frac{\|g\|}{4}\ge 2\theta_0 \|g\|,$$ and therefore \begin{equation}\label{equ:longTotalLength}
            Mk/2 > \theta_0 \|g\|.
        \end{equation} This inequality will allow us to apply Corollary~\ref{OGrowthTight}. Meanwhile, since $\|g\|>N = 16(4M+L)$, $$\frac{Mk}{4M + L} -2 - \theta \|g\| > \|g\|\left(\frac{1}{4(4M + L)} - \theta\right) -2 \ge \frac{\|g\|}{8(4M + L)} -2 \ge 0$$ which yields by the choice of $\theta$,   \begin{equation}\label{equ:largeGap}
            \frac{Mk}{4M + L} -2 > \theta \|g\|.
        \end{equation}
        This inequality will be used to deduce the $(\theta, L, M)$-linear recurrence for $g$.
        
        Let $K$ denote the set of those points in $\{x_i: 1\le i\le k\}$ that are contained in $N_{M}(Go)$. For notational simplicity, we re-index $K = \{y_1, y_2, \cdots, y_m\}$ with $m=|K|$.   
        We are ready to complete the proof by the following two claims.
 
        \begin{claim}
        If $m \ge \frac{k}{2}$ then $g\in\mathcal{D}_M(\theta, L)$.     
        \end{claim} 
        \begin{proof}[Proof of the Claim 1]
            Since $K\subseteq \{x_1, \cdots, x_k\}$, we have $d(y_i, y_j)>2M$ for all $0\le i<j\le m+1$, where we set $y_0 = o$ and $y_{m+1} = go$. In particular, for any $0\le i<j\le m+1$, we have $$d(y_i, y_j) = \sum_{l=i}^{j-1} d(y_i, y_{i+1}) >2M(j-i).$$

            Define $$N = \left\lceil\frac{L+2M}{2M}\right\rceil, \quad n =\left \lfloor\frac{m}{N}\right\rfloor - 1.$$ By the properties of the ceiling function, we have $$\frac{L+2M}{2M}\le N<\frac{L+4M}{2M}.$$ Now consider the subsequence $z_i = y_{i\cdot N}$ for $0\le i\le n$ and set $z_{n+1} = go$. For any $0\le i<j\le n+1$, because the number of $y$'s between $z_i$ and $z_j$ is at least $N$, we ahve $$d(z_i, z_j)>2MN > L+2M.$$ For each $1\le i\le n$, $z_i\in K$. Therefore, by definition of $K$, there exists $g_i\in G$ such that $d(g_io, z_i)\le M$. Applying the triangle inequality, for any $0\le i<j\le n+1$ we obtain $$d(g_io, g_jo)\ge d(z_i, z_j) - d(g_io, z_i) - d(g_jo, z_j)\ge (L+2M) - M-M = L.$$ It remains to estimate $n$ from below. Using $N<\frac{L+4M}{2M}$ and the definition of $n$, we have 
            $$n\ge \frac{m}{N}-2
            > \frac{2Mm}{L+4M} - 2\ge \frac{Mk}{L+4M}-2.$$ Together with inequality~\eqref{equ:largeGap}, this implies $$n\ge \theta \|g\|.$$ Thus, the points $\{z_1, \cdots, z_n\}$ and group elements $\{g_1, \cdots, g_n\}$ satisfy the $(\theta, L, M)$-linear recurrence for $g$, i.e., $g\in \mathcal{D}_M(\theta, L)$.  
        \end{proof}
        \begin{claim}
        If $m < \frac{k}{2}$ then $g\in \mathcal{O}_{M/2}({\theta_0, M}) \subseteq \mathcal{O}_{M_0}({\theta_0, L_0})$  
        \end{claim}    
        \begin{proof}[Proof of the Claim 2]
        
            For each $i$ with $1\le i\le k$ and $i\not\in K$, we have $x_i \not\in N_{M}(Go)$. By the triangle inequality, $$N_{M/2}(x_i) \subseteq X\setminus N_{M/2}(Go).$$
            
            Set $\alpha_i = N_{M/2}(x_i) \cap \gamma$. Then $\alpha_i$ has length $M$. Since $d(x_i, x_j) \ge 2M$ for $i\neq j$, the intervals $\alpha_i$ are pairwise disjoint for all $i\not\in K$. Let $\mathcal{K}$ be the collection of maximal connected components of length at least $M$ in $[o, go] \cap (X \setminus N_M(Go))$. Then each $\alpha_i$ is contained in some component of $\mathcal{K}$. Because the number of indices $i\notin K$ is at least $\frac{k}{2}$, the total length of $\mathcal{K}$ is at least $M k /2$. By inequality \eqref{equ:longTotalLength}, $Mk/2 > \theta_0\, l$. Hence $g\in \mathcal{O}_{M/2}({\theta_0, M}) \subseteq \mathcal{O}_{M_0}({\theta_0, L_0})$.
        \end{proof} 
        The two claims complete the proof.  
    \end{proof}

\subsection{A critical gap criterion}\label{GapCrit}
This subsection is devoted to Lemma \ref{CritGapLem} which is a key ingredient in the proof of Theorem \ref{MainThm}. This lemma provides a criterion for establishing growth tightness for a subset via an extension map.

The following elementary lemma will be useful.
\begin{lem}\label{addAlmGeo}
    Let $g = s_1 \cdots s_m$ be an $M$-almost geodesic product decomposition for some $M>0$. Consider a sequence of indices $i_0 := 0 < i_1 < i_2 \cdots < i_\alpha < m=:i_{\alpha+1}$ for some $1\le \alpha \le m$. Then $$\sum_{j=0}^{\alpha} d(g_{i_{j}}o, g_{i_{j+1}} o) \le d(o, go) + M \alpha.$$
\end{lem}
\begin{proof}
    We prove the statement by induction on $\alpha$. The case $\alpha = 0$ is trivial. Assume the lemma holds for $\alpha = \alpha_0 -1$ and consider $\alpha = \alpha_0$. By the definition of an $M$-almost geodesic decomposition applied to the indices $0,i_1, i_2$, we have $$d(o, g_{i_1} o) + d(g_{i_1} o, g_{i_2} o) \le d(o, g_{i_2} o) + M.$$ Applying the induction hypothesis to the sequence $i_2, i_3, \cdots, i_{\alpha}$ yields $$d(o, g_{i_2} o) + \sum_{j=2}^{\alpha} d(g_{i_j} o, g_{i_{j+1}} o) \le d(o, go) + M (\alpha - 1).$$ Adding these two inequalities gives $$\sum_{j=0}^{\alpha} d(g_{i_j} o, g_{i_{j+1}} o) \le d(o, go) + M \alpha.$$
    This completes the induction.
\end{proof}

Denote $B(n) = \{g\in G\mid d(o, go)\le n\}$ for $n\ge 0$. Recall that the \textit{growth rate} $\omega_G$ of a subset $A\subset G$ is defined as follows
		\[\omega_A = \limsup\limits_{n\to \infty} n^{-1} \ln |B(n)\cap A|\]
The Poincar\'e series of a subset $A\subset G$ is defined  as  
\begin{align}\label{eq:poincareseries}
\mathcal P(A,s) = \sum_{g\in A} \mathrm{e}^{-sd(o,go)}    
\end{align}
The growth rate $\omega_A$ of $A$ is the critical exponent of $\mathcal P(A,s)$: the series $\mathcal P(A,s)$ converges for $s>\omega_A$ and diverges for $s<\omega_A$. The action $G\act X$ is said to be \textit{of divergent type} if $\mathcal P(A,s)$ of $G$ diverges at $s=\omega_G$.

Recall that   $\str$ denotes the set of binary strings of length $m$, which can be identified with the power set of $\{0,1,2,\cdots, m-1\}$. See \textsection \ref{subsec:notation} for details.

\begin{lem}\label{CritGapLem}
Let $A\subseteq G$ be a subset, $\mathcal{F}$ be a finite set, $\theta\in (0,1)$, and $M>0$. Suppose the following conditions hold:

\begin{enumerate}
\item
 Each $g\in A$ admits an $M$-almost geodesic product decomposition $g=s_1\cdots s_m$ with $m\ge \theta d(o,go)$ (where $m$ may depend on $g$).
 \item
  For every $g\in A$, there exists an injective map $\Phi_g: \{0,1\}^m\to G$ such that for any $(\epsilon_1, \cdots, \epsilon_m)\in \str$, there exist elements $f_i\in \mathcal{F}$ (depending on $(\epsilon_i)$) satisfying $$\Phi_g({\epsilon_1},\cdots, {\epsilon_{m}})= \prod_{i=1}^m (f_i)^{\epsilon_i} s_i = (f_1)^{\epsilon_1}s_1 \cdots  (f_{m})^{\epsilon_{m}} s_m$$ where $m$ depends on $g$ given in (1).  
  \item
The images of the maps $\Phi_g$ are pairwise disjoint: for distinct $g, g'\in A$, we have $\Phi_g(\{0,1\}^m)\cap \Phi_{g'}(\{0,1\}^{m'})=\emptyset$, where $m,m'$ are the lengths of the decomposition associated to $g, g'$ respectively.
\end{enumerate} Then $\omega_G >\omega_A$. Moreover, the gap $\omega_G - \omega_A$ depends only on $\theta, M, \omega_G$ and $\max\{d(o,fo): f\in F\}$ (but not on $A$ itself).
 
\end{lem}
In this paper, we conventionally refer to $\Phi_g$ as the {\it extension map} of $g$.
\begin{proof}
   Set $R = M+\max\{d(o,fo): f\in F\}$. For $\overline{\epsilon} = (\epsilon_1, \cdots, \epsilon_m)\in \str$, let $|\overline{\epsilon}| = \sum_{i=1}^m\epsilon_i$. By the  assumption (1) and Lemma \ref{addAlmGeo}, we have 
   $$
   d(o, \Phi_{g}(\overline\epsilon) o) \le  d(o, go) + R |\overline \epsilon|.
   $$
   Using the disjointness condition (3), we obtain
   $$
    \begin{aligned}
        \mathcal P(G,s) &\ge \sum_{g\in A} \left( \sum_{\overline{\epsilon}\in \str} \mathrm{e}^{-s d(o, \Phi_{g}(\overline\epsilon)o)}\right) \\
        \\
        &\ge \sum_{g\in A} \left( \sum_{\overline{\epsilon}\in \str} \mathrm{e}^{-s d(o, go)-s R|\overline \epsilon|}\right) \\
        \\
        &= \sum_{g\in A}\left( \mathrm{e}^{-s d(o, go)} \sum_{\overline{\epsilon}\in \str} \mathrm{e}^{-sR\epsilon_1}\cdots \mathrm{e}^{-s R\epsilon_m}\right)
    \end{aligned}$$
    
    For any $x\in \mathbb{R}$, we have $$\displaystyle\sum_{\epsilon_1, \cdots, \epsilon_m\in\{0, 1\}} x^{\epsilon_1} \cdots x^{\epsilon_m} = (1 + x) ^ m.$$
    Since  $m\ge \theta\|g\|$, it follows that
    $$\begin{aligned}
        \mathcal P(G, s)
        & \ge \sum_{g\in A} \mathrm{e}^{-sd(o, go)} \left(1+\mathrm{e}^{-s R}\right)^m\\
        \\
        & \ge \sum_{g\in A} \mathrm{e}^{-sd(o, go)} \left(1+\mathrm{e}^{-s R}\right)^{\theta\|g\|}\\
        \\
        & = \sum_{g\in A} \left(\mathrm{e}^{-s} (1+\mathrm{e}^{-sR})^{\theta}\right)^{\|g\|}\\
        \\
        & = \mathcal P(A, \rho(s)) 
    \end{aligned}$$
    where  we define $$\rho(s) = s - \theta {\ln (1 + \mathrm{e}^{-s R})}.$$
    Then $\rho(s)<s$.

    \begin{claim}\label{GapLem}
        For any $0 < \omega \le \omega_G$, there exists $\omega' > \omega$ such that $\rho(\omega') < \omega$. Moreover, $\omega' - \omega$ depends only on $\theta$, $R$ and $\omega_G$ (but not on $A$).
    \end{claim}
    \begin{proof}
        Take $\omega' = \omega+\theta\ln(1 + \mathrm{e}^{-(\omega_G + 1) R}))$. Then $\omega' < \omega + 1 \le \omega_G + 1$, and consequently $$\rho(\omega') = \omega' - \theta {\ln (1 + \mathrm{e}^{-\omega' R})}< \omega' - \theta\ln(1 + \mathrm{e}^{-(\omega_G + 1) R}) = \omega.$$
    \end{proof}
Let $\omega = \omega_{A}$ be the growth rate of $A$. By the claim, we can choose $\omega'$ such that $\rho(\omega') < \omega < \omega'$. 

    Because $\rho(\omega') < \omega$, the series $\mathcal{P}(A, \rho(\omega'))$ diverges. Hence $\mathcal{P} (G, \omega')$ also diverges, which implies $\omega_G \ge \omega'$. Since $\omega'>\omega$ and $\omega' - \omega$ depends only on $\theta, R$ and $\omega_G$, we conclude $\omega_G>\omega_{A}$ and that the  gap $\omega_G - \omega_A$ is bounded below by a constant depending only on$\theta, R$ and $\omega_G$ (but not on $A$). This completes the  proof.   
\end{proof}

\section{Negligible quotient growth for confined subgroups}\label{SecNegliGrowth} 
The goal of this section is to prove Theorem \ref{NegliGrowthThm}. The key technical tool is a variant of the Extension Lemma \ref{extend3} adapted to the study of confined subgroups. This variant is essentially contained in \cite[Sec. 5.3]{CGYZ}, though we present it here in a slightly different form. In particular, our arguments do not rely on the boundary of the space.

\subsection{Extension Lemma for confined subgroups}\label{SSecCGYZ}
	
We fix a group $G$ that acts properly on a metric space $X$ with contracting elements. Define the elliptic radical \footnote{In   literature, $E(G)$ is sometimes called the finite radical. This terminology is perhaps more appropriate, as in the present setting it is always a finite group and does not depend on the actions (assuming the existence of one such action). We keep the terminology consistent with the one in \cite{CGYZ}, where the elliptic radical arises naturally from the action on boundary.}
 $$
 E(G):=\cap E(h)
 $$
 where the intersection is taken over all contracting elements $h\in G$, and $E(h)$ denotes the maximal elementary subgroup containing  $h$.
\begin{lem}\label{lem:E(G)}
$E(G)$ is the maximal finite normal subgroup in $G$. In particular, if $h$ is a contracting element, then
 $$
E(G)=\cap_{g\in G} gE(h)g^{-1} . 
 $$
\end{lem}

\begin{proof}
For a fixed contracting element $h$, the subgroup $K(h):=\bigcap_{g\in G} gE(h)g^{-1}$ is a normal subgroup of $G$.  We first show that any  finite normal subgroup $E$  of $G$ is contained in $K(h)$.  Indeed,  normality $Eh=hE$ implies that $E$ preserves $\langle h\rangle o$ up to finite Hausdorff distance, so $E\subseteq E(h)$. Since the set of contracting elements is invariant under conjugacy, we obtain $E\subseteq K(h)$.  

It remains  to prove that $K(h)$ is finite. If $K(h)$ were infinite, then by \cite[Lemma 4.6]{YANG6}, $K(h)$ would contain infinitely many pairwise independent contracting elements. Choose two independent contracting elements $k, f\in K(h)$. Then either $k$ or $f$ is independent of $h$. For concreteness, assume that $f$ and $h$ are independent. Independence implies that $E(f)\cap gE(h)g^{-1}$ is finite for every $g\in G$,  which contradicts the fact that $\langle f\rangle<K(h)$. Hence, $K(h)$ must be finite. 
\end{proof}
 
 Recall that a subset $P\subseteq G$ is called \textit{non-degenerate} if it is nonempty and disjoint from $E(G)$.
    \begin{lem}\label{lem:indepsetF}
    Let $P$  be a finite non-degenerate set. Then there exists a arbitrarily large finite set $F$ of independent contracting elements in $G$ such that for every $p\in P$ and every $f\in F$, we have $p \ax(f)\ne \ax(f)$.      
    \end{lem}
    \begin{proof}
    Given a contracting element $h$, Lemma \ref{lem:E(G)} yields  $P\cap E(G) = \emptyset$, where $E(G)=\cap_{g\in G} gE(h)g^{-1}$. Hence for each $p\in P$, there exists   $g\in G$ such that $pg\ax(h)\ne g\ax(h)$.  Since $G$ is  non-elementary group, it contains infinitely many independent contracting elements $h$. For each (and hence any) element $p$ in a finite set $P$, there exists an infinite set of  independent contracting elements $f_n$ satisfying $p \ax(f_n)\ne \ax(f_n)$ for all $n\ge 1$. Choosing a finite subfamily yields the desired set $F$. 
    \end{proof}

    The next lemma is a key observation used in the proof of Lemma \ref{extend3}. 
    \begin{lem}\label{lem:bddproj}
    Let $F$ be a finite set of independent contracting elements in $G$. Then there exists a constant $\tau$, depending on $F$, such that for any $g\in G$ and any two distinct $f_1, f_2\in F$, we have 
    $$
    \max\{\diam{\pi_{\ax(f_1)}([o,go])}, \diam{\pi_{\ax(f_2)}[o,go]} \}\le \tau.
    $$
    \end{lem} 

    From this we deduce the following variant of \cite[Lemma 5.7]{CGYZ}, adapted to confined subgroups with a common confining subset.    
	\begin{lem}\label{ExtLemConf}
		Let $P$ be a finite non-degenerate subset in $G$ (i.e. $P\cap E(G)=\emptyset$). Then there exist constants $L, \tau > 0$ and a finite subset $F\subset G$ of contracting elements  with the following property:
  
        For any confined subgroup $H\subset G$ admitting $P$ as a confining subset, and for any $g, h\in G$, there exists $f\in F$  and $p\in P$ such that $g f p f^{-1} g^{-1}\in H$ and the tuple $(g, f, p, f^{-1}, h)$ labels an $(L, \tau)$-admissible path.
	\end{lem}
	\begin{proof}
	Let $F$ be a fixed set of contracting elements obtained from Lemma \ref{lem:indepsetF}. Thus, for all $p\in P$ and $f\in F$, we have $p \ax(f)\ne \ax(f)$. Given any $g, h\in G$, choose $f\in F$ such that $$\max\{\pi_{\ax(f)}([o,go]),\pi_{\ax(f)}([o,ho])\}\le \tau.$$ Since $P$ is a confining subset for $H$, we can select $p\in P$ such that $g f p f^{-1} g^{-1}\in H$. Setting $L=\min\{d(o,fo): f\in F\}$, a routine verification (see \cite[Lemma 5.7]{CGYZ} for details) shows that $(g, f, p, f^{-1}, h)$ labels an $(L, \tau)$-admissible path. 
	\end{proof}
 As an aside of independent interest, we explain an alternative proof of $C^{\ast}$-algebra of $G$ with trivial $E(G)$ via the URS characterization of simple $C^{\ast}$-algebra in \cite{Ken20}.
 In \cite{CGYZ}, it is proved that a confined subgroup with non-degenerating confined subset contains independent contracting elements, so contains a  non-abelian free group. Consequently, if $E(G)$ is trivial, then $G$ contains no nontrivial URS. The recent work of \cite{Ken20} implies that the reduced $C^{\ast}$-algebra of $G$ is simple. This fact is  well-known  for  any acylindrically hyperbolic groups without nontrivial finite normal subgroups  (see \cite{DGO}).

\subsection{Negligible growth for confined subgroups: proper action}
We are ready to prove Theorem \ref{NegliGrowthThm}, which is an immediate consequence of the more general result in a proper action.
\begin{thm}\label{NegligibleGrowthThm}
Assume that $G$ acts properly on a geodesic metric space $X$ with a contracting element. Let $H$ be a confined subgroup with a finite non-degenerate confining subset. Then
$$\frac{|B_{G/H}(n)|}{\mathrm{e}^{n\omega_G}}\to 0,\;\;\text{as}\;\; n\to \infty $$
\end{thm}

The rest of this section is devoted to the proof of this theorem.

Let $A$ be a complete set of shortest representatives for the cosets $Hg\in G/H$; that is, $A$ contains exactly one element $g\in Hg$ for each coset $Hg\in G/H$ and satisfies $d(o,go)=d(o, Hgo)$. By definition, $|B_{G/H}(n)|=|A\cap B(n)|$.

Choose a maximal $R$-separated subset $\tilde A$ of $A$ for some $R>0$. It is known (see \cite[Lemma 2.24]{YANG10}) that there exists $\theta>0$, depending only on $R$, such that $|\tilde A\cap B(n)|>\theta |A\cap B(n)|$ for all large $n$. Thus, to prove Theorem \ref{NegligibleGrowthThm}, it suffices to show  $$\frac{|\tilde A\cap B(n)|}{\mathrm{e}^{n\omega_G}}\to 0.$$ The following lemma is crucial.

\begin{lem}
Let $F$ and $P$  be finite sets provided by Lemma \ref{ExtLemConf}. Denote $\mathcal{W}=\cup_{n\ge 1}(\tilde A)^n$.  Then if  $R$ is chosen sufficiently large,   the map $\Phi: \mathcal{W} \to G$ defined by
$$
\begin{aligned}
(a_1, \cdots, a_n) \longmapsto a_1 f_1 p_1 f_1^{-1} \cdots a_n f_n p_n f_n^{-1},
\end{aligned}
$$
where $f_i,p_i$ are chosen by Lemma \ref{ExtLemConf} for the pair $(a_i,a_i)$, is injective.
\end{lem}
\begin{proof}
By taking sufficiently high powers of elements in $F$ and $P$, we may assume  $$R_0=\min\{d(o,fpf^{-1}o): f\in F,\ p\in P\}>8r$$
where $r$ is the fellow travel constant associated to $(L,\tau)$-admissible paths by Proposition \ref{admisProp}.

We claim that taking $R\ge R_0$ works. Suppose $\Phi(a_1, \cdots, a_n)=\Phi(a_1', \cdots, a_m')$. By Lemma \ref{ExtLemConf}, both sides give $(L,\tau)$-admissible paths with the same endpoints . Hence they $r$-fellow travel a common geodesic $\alpha$.

We are going to show that $m=n$ and $a_i=a_i'$ for $1\le i\le m$. Assume, for contradiction, that $a_1\ne a_1'$. Without loss, suppose $d(o, a_1'o)\ge d(o,a_1o)$. By the choice of $f_1, p_1$, we have 
$a_1 f_1 p_1 f_1^{-1}a_1^{-1}=h_1\in H\setminus \{1\}$, so $a_1 f_1 p_1 f_1^{-1}=h_1a_1$. 

The $r$-fellow travel property implies $$d(a_1o, [o,a_1'o]),\qquad d(a_1f_1 p_1 f_1^{-1}o, [o,a_1'o])\le 2r.$$ This yields the first line by triangle inequality: 
$$
\begin{aligned}
d(o, a_1'o)+8r&\ge  d(o, a_1o)+d(o,f_1 p_1 f_1^{-1}o) + d(h_1a_1o, a_1'o) \\
&\ge  d(o, a_1o)+ d(h_1a_1o, a_1'o) +R_0
\end{aligned}
$$
Since $a_1, a_1'\in A$ are shortest representatives: $d(o, Ha_1o)=d(o, a_1o)$, $d(o, Ha_1'o)=d(o, a_1'o)$. Consequently,
$$
\begin{aligned}
d(o, Ha_1'o)=d(o, a_1'o)&\ge d(o, a_1o) + d(h_1a_1o, a_1'o)+R_0 -8r\\
&\ge d(o, a_1o) + d(a_1o, h_1^{-1}a_1'o) +R_0-8r\\
& \ge d(o, Ha_1'o)+R_0-8r>d(o, Ha_1'o),
\end{aligned}
$$
a contradiction. Hence $a_1 = a_1'$, and by induction we obtain $n=m$ and $a_i = a_i'$ for all $i$. 
\end{proof}

By \cite[Lemma 2.23]{YANG10}, the Poincar\'e series associated to $\tilde A$ converges at the growth rate $\omega_{\Phi(\mathcal W)}$ of $\Phi(\mathcal W)$, which is at most  $\omega_G$. Therefore,  
$$
\sum_{a\in \tilde  A} \mathrm{e}^{-\omega_G d(o,ao)} \asymp \sum_{n\ge 1} |\tilde  A\cap B(n)| \mathrm{e}^{-\omega_G n} <\infty
$$
Hence $|\tilde  A\cap B(n)| \mathrm{e}^{-\omega_G n}\to 0$.
Since $|\tilde  A\cap B(n)|\ge \theta |A\cap B(n)|$ for some $\theta>0$, this completes the proof of Theorem \ref{NegligibleGrowthThm}.

\section{Growth tightness for Confined Subgroups in SCC Actions}\label{SecGrowthTight}

Throughout this section, we assume that $G$ admits  a proper action on a geodesic metric space $X$. We fix a finite  non-degenerate set   $P\subset G$ (that is, $P\cap E(G)=\emptyset$). 

Let $H$ be a confined subgroup of $G$ with $P$ as a confining subset. Let $[G/H]\subset G$ denote any section of the natural map $$G\to G/H=\{Hg: g\in G\}$$ which picks up exactly one element from each $Hg$. Given $\theta, L, M>0$, denote $$A_{\theta, L, M} =[G/H]\cap \mathcal{D}_{M}(\theta, L).$$ 
Here $\mathcal{D}_M(\theta, L)$ denotes the set of $(\theta, L, M)$-linearly recurrent elements in $G$.
In plain words, this is a set of $H$-right coset representatives $g\in G$ that admit an $M$-almost geodesic product decomposition with $m\ge \theta d(o,go)$: $$g=s_1s_2\cdots s_m$$ has the property that the   orbital points $g_io\in Go$ (with $g_0=1$ and $g_i=s_1\cdots s_i$ for $1\le i\le m$) are at distance at least $L$ from each other (see Definition \ref{QconvexProperty}).

The main result of this section is the following theorem, from which Theorem~\ref{MainThm} will be deduced.
    \begin{thm}\label{AGroTig}
        For any $\theta\in (0,1], M>0$ there exists $L = L(M)$ such that $A := A_{\theta, L, M}$ is growth tight. Moreover, the gap $\omega_G-\omega_A>0$ depends only on $\theta,M, \omega_G$ and $P$ (but \emph{not} on $H$).
    \end{thm}

Let $g\in A$ be as above. The core of the proof is to define a map $$\Phi_g: \str\to G$$ (where $\str$ denotes the set of binary strings of length $m$, identified with subsets of $\{0, 1, \cdots, m-1\}$, see \textsection \ref{subsec:notation}) with the following properties:
\begin{itemize}
    \item $\Phi_g$ is injective (Lemma \ref{PhiInj});
    \item $\operatorname{Im}(\Phi_g) \subseteq Hg$ (Lemma \ref{MaptoCoset});
    \item For every $\bar\epsilon\in \str$, the element $\Phi_g(\bar\epsilon)$ labels an admissible path (Lemma \ref{AugPathAdm}). 
\end{itemize} 
These properties will ensure that the conditions of Lemma \ref{CritGapLem} are satisfied, yielding the desired growth gap. 

We begin by clarifying the logical dependencies among the constants involved.

    \subsection{Choice of Constants}\label{ChoCon}

    Given $P$, let   $F$  be a finite set of contracting elements provided by  Lemma \ref{ExtLemConf}, whose axes $\ax(f)$ for $f\in F$ are $C$-contracting for some $C>0$. Let $(L_0, \tau_0)$ be the constants for admissible paths given by  Lemma \ref{ExtLemConf}. We emphasize that $F$ depends only on $P$, but not on $H$; this is essential for ensuring that the gap in Theorem \ref{AGroTig} is independent of $H$.

    Set $\tau = \tau_0+3M + 13C$ (given by Lemma \ref{AugPathAdm}). We then choose $L = L(\tau_0, C, M)$ large enough so that the following hold:
    \begin{enumerate} 
        \item[\textbf{(L1)}] $L>\tau$;
        \item[\textbf{(L2)}]  \label{property:admissible} There exists $r = r(\tau_0, C, M)>0$ such that any $(L, \tau)$-admissible path has the $r$-fellow travel property (see Proposition \ref{admisProp});
        \item[\textbf{(L3)}]\label{property:injective} $L > 4N$, where $N = N(\tau, C)$ is the constant given by Corollary~\ref{ProjMidSet}.
    \end{enumerate}
    On replacing each $f\in F$ with a sufficiently high powers $f^n$, we may assume that $d(o, f o)> L$ for every element $f\in F$. Note that this does not affect the $C$-contracting constant of $\ax(f)$ for each $f\in F$. 
    \subsection{Construction of the map $\Phi_g$}\label{ConsMap}
    Let $g=s_1s_2\cdots s_m\in G$ be  a product decomposition. At this stage, we do not yet assume that the decomposition is $M$–almost geodesic or that $m \ge \theta d(o,go)$.  Set $g_0=1$ and $g_i=s_1\cdots s_i$ for $1\le i\le m$. Then for $0\le i< j\le m$, we have
\[
g_i^{-1}g_j = s_{i+1}s_{i+2}\cdots s_j.
\]
By Lemma~\ref{ExtLemConf}, for each $k$ with $0\le k\le m-1$ there exist $f_k\in F$ and $p_k\in P$ such that 
\begin{equation}
\tag{$\circledast$}\label{eq:gkconjugates}
\begin{array}{lr}
\text{the element }  g_{k}f_kp_kf_k^{-1}g_{k}^{-1}\in H, \text{ and } & \\
\text{the word } (g_{k},\, f_k,\, p_k,\, f_k^{-1},\, g_{k}^{-1}g_m) \text{ labels an } (L,\tau_0) \text{–admissible path.} & 
\end{array}
\end{equation} 
Note that $m$ depends on the specific element $g$ and is not required to be related to the word length $\|g\|$ at this point.

\begin{defn}\label{def:extMap}
Let $I\subseteq \{0,1,2,\cdots, m-1\}$ be a non-empty subset, enumerated in increasing order as
\[
I=\{i_1<i_2<\dots<i_\alpha\}, \qquad \alpha=|I|.
\]
Define the map $\Phi_g: \str \to G$ by
\[
\Phi_g(I)
= g_{i_1}(f_{i_1}p_{i_1}f_{i_1}^{-1})\, g_{i_1}^{-1}g_{i_2}(f_{i_2}p_{i_2}f_{i_2}^{-1})
\cdots
(f_{i_\alpha}p_{i_\alpha}f_{i_\alpha}^{-1})\, g_{i_\alpha}^{-1}g
\]  
and $\Phi_g(\varnothing) = g$.
\end{defn}
\begin{rem}[Alternative definition]\label{rem:alterDef}
Identify $I$ with a binary string $(\epsilon_0,\dots,\epsilon_{m-1})\in\str$ (where $\epsilon_j = 1$ iff $j\in I$), we can write more compactly that
\[
\Phi_g(I)=\prod_{j=0}^{m-1} (f_{j} p_{j} f_{j}^{-1})^{\epsilon_{j}}\,s_{j+1}.
\]   
That is,  $\Phi_g(I)$ inserts the element $(f_{j} p_{j} f_{j}^{-1})$ before the letter $s_{j+1}$ exactly when $\epsilon_j = 1$. See Figure \ref{fig:ProfProjHas} for the illustration. 
\end{rem}

The following result is immediate from the construction.
\begin{lem}\label{MaptoCoset}
For any product decomposition $g=s_1s_2\cdots s_m\in G$ with $m\ge 1$, we have $$\Phi_g(\str)\subseteq Hg.$$
\end{lem}
\begin{proof}
Take $g\in A$ and 
\[
I=\{i_1,i_2,\dots,i_\alpha\}, \quad i_j<i_{j+1},\; 1\le j<\alpha=|I|
\] as above.
By associativity,
\[\begin{aligned}
    \Phi_g(I)
&= g_{i_1}(f_{i_1}p_{i_1}f_{i_1}^{-1})\, g_{i_1}^{-1}g_{i_2}(f_{i_2}p_{i_2}f_{i_2}^{-1})
\cdots
(f_{i_\alpha}p_{i_\alpha}f_{i_\alpha}^{-1})\, g_{i_\alpha}^{-1}g\\
&=\prod_{j=1}^{\alpha}(g_{i_j}f_{i_j}p_{i_j}f_{i_j}^{-1}g_{i_j}^{-1}) g.
\end{aligned}
\]
Condition~\eqref{eq:gkconjugates} guarantees that each factor $g_{i_j}f_{i_j}p_{i_j}f_{i_j}^{-1}g_{i_j}^{-1}\in H$. Since $H$ is a subgroup, we have their product also belongs to $H$. Hence $\Phi_g(I)\in Hg$.
\end{proof}

As a corollary, the property (3) of Lemma \ref{CritGapLem} holds: for distinct $g, g'\in A$, the cosets $Hg$ and $ Hg'$ are disjoint , so $\Phi_g(\str) \cap \Phi_{g'}(\{0,1\}^{m'}) = \emptyset$. 
     
\subsection{Admissible paths associated to elements $\Phi_g(I)$}\label{AugIsAdm}
From now on, we assume $$g\in A_{\theta, L, M} =[G/H]\cap \mathcal{D}_{M}(\theta, L)$$ which admits a  decomposition $g=s_1s_2\cdots s_m$ with 
\[
d(g_i o, [o,go])\le M
\quad\text{and}\quad
d(g_i o,g_{i+1}o)>L,
\]
We define the map $\Phi_g: \str\to Hg$ as above and study the path $\gamma_g(I)$ labeled by the tuple
\[
\bigl(
g_{i_1}, f_{i_1}, p_{i_1}, f_{i_1}^{-1}, g_{i_1}^{-1}g_{i_2}, f_{i_2}, p_{i_2}, f_{i_2}^{-1},
\ldots,
f_{i_\alpha}, p_{i_\alpha}, f_{i_\alpha}^{-1}, g_{i_\alpha}^{-1}g
\bigr),
\]
where $I = \{i_1, \cdots, i_{\alpha}\}$.
Then $\gamma_g(I)$ is a concatenation of geodesics joining $o$ to $\gamma_g(I)o$. Lemma~\ref{AugPathAdm} will show that $\gamma_g(I)$ is an admissible path; consequently, by Proposition~\ref{admisProp}, it fellow travels any geodesic segment from $o$ to $\gamma_g(I)o$.

    The next three preparatory lemmas concern projections onto a $C$-contracting subset $Y\subset X$ for some $C>0$. 
    \begin{lem}\label{NearEndPoint}
        Let $x, y, z\in X$ with $d(y, z) < M$ for some $M>0$. Then $$\bigl|\proj_Y(x,y) - \proj_Y(x,z)\bigr| < M + 2C.$$
    \end{lem}
    \begin{proof}
By Lemma~\ref{BigFive}(2), $\proj_Y(z,y)\le d(z,y)+2C< M+2C$. The triangle inequality for $\proj_Y$ gives
\[
\proj_Y(x,z)\le \proj_Y(x,y)+M+2C.
\]
Swapping $y$ and $z$ yields the reverse bound.
\end{proof}

    \begin{lem}\label{lem:ProjInGeod}
Let $x,y\in X$ and let $z\in [x,y]$. Then 
\[
\proj_Y(x,z)\le \proj_Y(x,y)+3C.
\]
\end{lem}

\begin{proof}
If $d(x,z)\le d(x,Y)$, the definition of $C$–contracting property (Definition~\ref{Def:Contracting}) gives $\proj_Y(x,z)\le C.$
If $d(y,z)\le d(y,Y)$, again the $C$–contracting property yields $\proj_Y(y,z)\le C.$
The triangle inequality for $\proj_Y$ gives,
\[
\proj_Y(x,z)
\le \proj_Y(x,y)+\proj_Y(y,z)
\le \proj_Y(x,y)+C.
\]
In these two cases, the proof is completed. 

Assume now that $d(x,z)> d(x,Y)$ and $d(y,z)> d(y,Y)$. 
Since $z\in[x,y]$, we have
$d(x,z)+d(z,y)=d(x,y).$
By the triangle inequality,
\[
d(x,y)\le d(x,Y)+d(y,Y)+\proj_Y(x,y).
\]
Combining these gives
\[
d(x,z)-d(x,Y)\le \proj_Y(x,y)-(d(z,y)-d(y,Y))\le \proj_Y(x,y).
\]
where we used the assumption $d(y,z)> d(y,Y)$.

Choose $w\in[x,z]$ with $d(x,w)=d(x,Y)$. Then and
\[
d(w,z)=d(x,z)-d(x,w)\le \proj_Y(x,y).
\] By the contracting property,
\[
\proj_Y(x,w)\le C,
\]

Lemma~\ref{BigFive}(2) gives
\[
\proj_Y(w,z)\le \proj_Y(x,y)+2C.
\]
Applying the triangle inequality,
\[
\proj_Y(x,z)
\le \proj_Y(x,w)+\proj_Y(w,z)
\le C+\proj_Y(x,y)+2C
=\proj_Y(x,y) + 3C.
\]
\end{proof}

    \begin{lem}\label{TruncAdmPath}
Let $x,y,z\in X$ with $d\bigl(z,[x,y]\bigr)<M$ for some $M>0$.  Then
\[
\proj_Y(x,z)\le \proj_Y(x,y)+ M + 5C.
\]
\end{lem}

\begin{proof}
Choose $v\in[x,y]$ with $d(v,z)<M$. By Lemma~\ref{lem:ProjInGeod}, we have
\[
\proj_Y(x,v)\le \proj_Y(x,y)+3C.
\]
Since $d(v,z)<M$, Lemma~\ref{NearEndPoint} implies that there exists $M_1=M_1(M,C)$ such that
\[
\proj_Y(x,z)\le \proj_Y(x,v)+M+2C.
\]
Combining the two inequalities yields
\[
\proj_Y(x,z)\le \proj_Y(x,y)+M+5C.
\]
The proof is complete.
\end{proof}
    We are now ready to show that $\gamma_g(I)$ is admissible. Recall that $F$ consists of finitely many contracting elements with $C$–contracting axes.

\begin{lem}\label{AugPathAdm}
There exists $\tau=\tau(\tau_0, C,M)>0$ such that the following holds: for every nonempty set $I\in\str$ and  $g\in A$, the path $\gamma_g(I)$ is $(L,\tau)$–admissible.
\end{lem}

\begin{proof}
Let $\gamma=[o,go]$ and $I = \{i_1, \cdots, i_\alpha\}$. Recall that $g\in \mathcal{D}_M(\theta,L)$ provides a product decomposition $g=s_1\cdots s_m$ with
\[
d(g_i o,\gamma)\le M
\quad\text{and}\quad
d(g_i o,g_{i+1}o)>L,
\]
where $g_i=s_1\cdots s_i$ for $1\le i\le m$ and $g_0=1$. The proof uses only these two properties, and the condition $m \ge \theta d(o,go)$ is unnecessary here. 

Let $\tau=\tau_0+3M + 13C$. Define
\[
X_k=\mathrm{Ax}(f_{i_k}) \quad (1\le k\le \alpha).
\]
According to Definition~\ref{AdmDef}, we must verify the following for each $1\le k\le \alpha$:
\begin{enumerate}
\item $\diam{\pi_{X_k}([o,(g_{i_{k-1}}^{-1}g_{i_k})^{-1}o])}\le \tau$;
\item $\diam{\pi_{X_k}([o,(g_{i_k}^{-1}g_{i_{k+1}})o])}\le \tau$;
\item $X_k\neq p_k X_k$;
\item $g_{i_k}X_k\neq g_{i_{k+1}}X_{k+1}$ whenever $k<\alpha$.
\end{enumerate}
Here $i_0=0$ and $i_{\alpha+1}=m+1$.

\medskip
\noindent\textit{Verification of (1).}
Set $Y=X_k$, $h=g_{i_k}$, $h_1=g_{i_{k-1}}$, and $h_2=g_{i_{k-1}}^{-1}g_{i_k}$, so that $h=h_1h_2$ and the property (1) is equivalent to $\diam{\pi_Y([o,h_2^{-1}o])}\le \tau.$

By the defining properties of $f_{i_k},p_{i_k}$ in (\ref{eq:gkconjugates}), the word
\[
(h,f_{i_k},p_{i_k},f_{i_k}^{-1},g_{i_k}^{-1}g)
\]
labels an $(L_0,\tau_0)$–admissible path, which by Definition \ref{AdmDef} implies
\[
\diam{\pi_Y([h^{-1}o,o])}\le \tau_0.
\]

Since $ho,h_1o\in N_M(\gamma)$, let us choose $y,y_1\in\gamma$ with $
d(ho,y)\le M,\, d(h_1o,y_1)\le M.$
Then
\[
d(o,h^{-1}y)=d(ho,y)\le M,\qquad
d(h_2^{-1}o,h^{-1}y_1)=d(h_1o,y_1)\le M.
\]
Because $h$ acts isometrically,
\[
\diam{\pi_Y(h^{-1}[o,y])}
=\diam{\pi_Y([h^{-1}o,h^{-1}y])}.
\]
By Lemma~\ref{NearEndPoint},
\[
\proj_Y(h^{-1}o, h^{-1}y)
\le \proj_Y(h^{-1}o,o)+M+2D
\le \tau_0+M+2D.
\]

Since $o,y_1,y$ lie in this order along $\gamma$, we have $y_1\in[o,y]$, and hence
\[
d(h_2^{-1}o,h^{-1}[o,y])\le M.
\]
By Lemma~\ref{TruncAdmPath},
\[
\proj_Y(h_2^{-1}o,h^{-1}y)
\le \proj_Y(h^{-1}o,h^{-1}y)+M + 5D
\le \tau_0+ 2M + 7D.
\]
Applying Lemma~\ref{NearEndPoint} once more,
\[
\proj_Y(h_2^{-1}o,o)
\le \proj_Y(h_2^{-1}o,h^{-1}y)+M + 2D
\le \tau_0+3M + 9D.
\]
By Lemma~\ref{BigFive}(1),
\[
\diam{\pi_Y([o, h_2^{-1} o])} 
\le \proj_Y(o, h_2^{-1} o) + 4D
\le \tau_0 + 3M + 13D 
= \tau.
\]

\medskip
\noindent\textit{Verification of (2).}
This follows by applying the same argument as above after replacing $g$ with $g^{-1}$ and reversing indices; the estimates are identical.

\medskip
\noindent\textit{Verification of (3).}
By the defining property of $f_k, p_k$ in (\ref{eq:gkconjugates}),  $\mathrm{Ax}(f_k)\neq p_k\mathrm{Ax}(f_k)$ follows from the definition \ref{AdmDef} of admissible paths.

\medskip
\noindent\textit{Verification of (4).}
By the choice  of (L1)  in \textsection \ref{ChoCon}, $L>\tau$,  and
\[
d(o,g_{i_k}^{-1}g_{i_{k+1}}o)=d(g_{i_k}o,g_{i_{k+1}}o)\ge L,
\]
the claim follows immediately from Lemma~\ref{LL2ImplyBP}.
\end{proof}

    By Lemma \ref{AugPathAdm}, $\gamma_g(I)$ is $(L, \tau)$-admissible. By the choice  of (L2)  in \textsection \ref{ChoCon}, any $(L, \tau)$-admissible path has $r$-fellow travel property for some $r = r(\tau, D) > 0$. We therefore obtain the following corollary.
    \begin{cor}\label{AugPathFel}
     $\gamma_g(I)$ has $r$-fellow travel property for some $r = r(\tau_0, C, M)$.
    \end{cor}

    \subsection{Injectivity of the map $\Phi_g$}
    The goal of  this subsection is to verify the property (2) of Lemma \ref{CritGapLem}. Namely,
    \begin{lem}\label{PhiInj}
        Let $ g\in A_{\theta, L, M} =[G/H]\cap \mathcal{D}_{M}(\theta, L)$ with a product decomposition $g=s_1s_2\cdots s_m$ as above. Then $\Phi_g: \str\to Hg$ is an injective map.
    \end{lem}

    Let $I\in \str$ be a non-empty subset of $\{0,1,\cdots, m-1\}$. Given $0\le k\le m-1$, we denote 
    $$h_k := \prod_{j=0}^{k-1} (f_{j} p_{j} f_{j}^{-1})^{\epsilon_{j}}\,s_{j+1}.$$ and $Y:= h_k \ax(f_k)$.  Denote $h=\Phi_g(I)$.
   We obtain opposing bounds on $\proj_{Y}(o, h o)$ according to whether $k \in I$ or $k \notin I$, which shall be used in arriving a contradiction in the proof of  Lemma \ref{PhiInj}.

    Let $N = N(\tau, C)$ be the constant given in Corollary~\ref{ProjMidSet} for $(L, \tau)$-admissible paths.
    \begin{lem}\label{ProjHas}
        Assume $k\in I$. Then $\proj_Y(o, ho) > L - 2N.$
    \end{lem}
    \begin{proof}
        \begin{figure}[h]
            \centering
            \tikzset{every picture/.style={line width=0.75pt}} 
            
            \begin{tikzpicture}[x=0.75pt,y=0.75pt,yscale=-1,xscale=1]
            
            \draw    (44,171) -- (44.98,80.05) ;
            \draw [shift={(45,78.05)}, rotate = 90.62] [color={rgb, 255:red, 0; green, 0; blue, 0 }  ][line width=0.75]    (10.93,-3.29) .. controls (6.95,-1.4) and (3.31,-0.3) .. (0,0) .. controls (3.31,0.3) and (6.95,1.4) .. (10.93,3.29)   [fill={rgb, 255:red, 0; green, 0; blue, 0 }  ](0, 0) circle [x radius= 3.35, y radius= 3.35]   ;
            \draw [shift={(44,171)}, rotate = 0] [color={rgb, 255:red, 0; green, 0; blue, 0 }  ][fill={rgb, 255:red, 0; green, 0; blue, 0 }  ][line width=0.75]      (0, 0) circle [x radius= 3.35, y radius= 3.35]   ;
            \draw    (45,78.05) -- (118,78.05) ;
            \draw [shift={(120,78.05)}, rotate = 180] [color={rgb, 255:red, 0; green, 0; blue, 0 }  ][line width=0.75]    (10.93,-3.29) .. controls (6.95,-1.4) and (3.31,-0.3) .. (0,0) .. controls (3.31,0.3) and (6.95,1.4) .. (10.93,3.29)   [fill={rgb, 255:red, 0; green, 0; blue, 0 }  ](0, 0) circle [x radius= 3.35, y radius= 3.35]   ;
            \draw    (120,78.05) -- (120,171) ;
            \draw [shift={(119,171)}, rotate = 270.61] [color={rgb, 255:red, 0; green, 0; blue, 0 }  ][line width=0.75]    (10.93,-3.29) .. controls (6.95,-1.4) and (3.31,-0.3) .. (0,0) .. controls (3.31,0.3) and (6.95,1.4) .. (10.93,3.29)   [fill={rgb, 255:red, 0; green, 0; blue, 0 }  ](0, 0) circle [x radius= 3.35, y radius= 3.35]   ;
            \draw    (119,171) -- (187,171) ;
            \draw [shift={(189,171)}, rotate = 180] [color={rgb, 255:red, 0; green, 0; blue, 0 }  ][line width=0.75]    (10.93,-3.29) .. controls (6.95,-1.4) and (3.31,-0.3) .. (0,0) .. controls (3.31,0.3) and (6.95,1.4) .. (10.93,3.29)   [fill={rgb, 255:red, 0; green, 0; blue, 0 }  ](0, 0) circle [x radius= 3.35, y radius= 3.35]   ;
            \draw    (244,171) -- (307,171) ;
            \draw [shift={(309,171)}, rotate = 180] [color={rgb, 255:red, 0; green, 0; blue, 0 }  ][line width=0.75]    (10.93,-3.29) .. controls (6.95,-1.4) and (3.31,-0.3) .. (0,0) .. controls (3.31,0.3) and (6.95,1.4) .. (10.93,3.29)   [fill={rgb, 255:red, 0; green, 0; blue, 0 }  ](0, 0) circle [x radius= 3.35, y radius= 3.35]   ;
            \draw [shift={(244,171)}, rotate = 0] [color={rgb, 255:red, 0; green, 0; blue, 0 }  ][fill={rgb, 255:red, 0; green, 0; blue, 0 }  ][line width=0.75]      (0, 0) circle [x radius= 3.35, y radius= 3.35]   ;
            \draw    (309,171) -- (309,82.05) ;
            \draw [shift={(309,80.05)}, rotate = 90] [color={rgb, 255:red, 0; green, 0; blue, 0 }  ][line width=0.75]    (10.93,-3.29) .. controls (6.95,-1.4) and (3.31,-0.3) .. (0,0) .. controls (3.31,0.3) and (6.95,1.4) .. (10.93,3.29)   [fill={rgb, 255:red, 0; green, 0; blue, 0 }  ](0, 0) circle [x radius= 3.35, y radius= 3.35]   ;
            \draw    (309,80.05) -- (377,80.05) ;
            \draw [shift={(379,80.05)}, rotate = 180] [color={rgb, 255:red, 0; green, 0; blue, 0 }  ][line width=0.75]    (10.93,-3.29) .. controls (6.95,-1.4) and (3.31,-0.3) .. (0,0) .. controls (3.31,0.3) and (6.95,1.4) .. (10.93,3.29)         [fill={rgb, 255:red, 0; green, 0; blue, 0 }  ](0, 0) circle [x radius= 3.35, y radius= 3.35]   ;
            \draw    (379,80.05) -- (379,169.05) ;
            \draw [shift={(379,171)}, rotate = 270] [color={rgb, 255:red, 0; green, 0; blue, 0 }  ][line width=0.75]    (10.93,-3.29) .. controls (6.95,-1.4) and (3.31,-0.3) .. (0,0) .. controls (3.31,0.3) and (6.95,1.4) .. (10.93,3.29)   [fill={rgb, 255:red, 0; green, 0; blue, 0 }  ](0, 0) circle [x radius= 3.35, y radius= 3.35]   ;
            \draw    (379,171) -- (453,171) ;
            \draw [shift={(455,171)}, rotate = 180] [color={rgb, 255:red, 0; green, 0; blue, 0 }  ][line width=0.75]    (10.93,-3.29) .. controls (6.95,-1.4) and (3.31,-0.3) .. (0,0) .. controls (3.31,0.3) and (6.95,1.4) .. (10.93,3.29)   [fill={rgb, 255:red, 0; green, 0; blue, 0 }  ](0, 0) circle [x radius= 3.35, y radius= 3.35]   ;
            \draw[dashed]   (274,125.52) .. controls (274,100.41) and (289.67,80.05) .. (309,80.05) .. controls (328.33,80.05) and (344,100.41) .. (344,125.52) .. controls (344,150.64) and (328.33,171) .. (309,171) .. controls (289.67,171) and (274,150.64) .. (274,125.52) -- cycle ;
            
            \draw (17,115) node [anchor=north west][inner sep=0.75pt]   [align=left] {$\displaystyle f_{0}$};
            \draw (201,171) node [anchor=north west][inner sep=0.75pt]   [align=left] {$\displaystyle \cdots $};
            \draw (125,114) node [anchor=north west][inner sep=0.75pt]   [align=left] {$\displaystyle f_{0}^{-1}$};
            \draw (70,64) node [anchor=north west][inner sep=0.75pt]   [align=left] {$\displaystyle p_{0}$};
            \draw (139,174) node [anchor=north west][inner sep=0.75pt]   [align=left] {$\displaystyle s_{1}$};
            \draw (255,174) node [anchor=north west][inner sep=0.75pt]   [align=left] {$\displaystyle s_{k}$};
            \draw (287,116) node [anchor=north west][inner sep=0.75pt]   [align=left] {$\displaystyle f_{k}$};
            \draw (299,67) node [anchor=north west][inner sep=0.75pt]   [align=left] {$\displaystyle vo$};
            \draw (332,67) node [anchor=north west][inner sep=0.75pt]   [align=left] {$\displaystyle p_{k}$};
            \draw (380,116) node [anchor=north west][inner sep=0.75pt]   [align=left] {$\displaystyle f_{k}^{-1}$};
            \draw (248,107) node [anchor=north west][inner sep=0.75pt]   [align=left] {$\displaystyle Y_{k}$};
            \draw (37,174) node [anchor=north west][inner sep=0.75pt]   [align=left] {$\displaystyle o$};
            \draw (299,174) node [anchor=north west][inner sep=0.75pt]   [align=left] {$\displaystyle u o$};
            \draw (408,173) node [anchor=north west][inner sep=0.75pt]   [align=left] {$\displaystyle s_{k+1}$};
            \draw (448,176) node [anchor=north west][inner sep=0.75pt]   [align=left] {$\displaystyle ho$};

            \end{tikzpicture}
            \caption{Proof of Lemma \ref{ProjHas}}
            \label{fig:ProfProjHas}
        \end{figure}
        Let $u = h_k$, $v = h_k f_k$. Note that $[u o, v o]$ is the geodesic segment in the admissible path $\gamma_g(I)$ whose endpoints are in the contracting subset $Y$. 
        By Corollary \ref{AugPathFel}, there exists $r = r(\tau_0, C, M)$ such that $\gamma_g(I)$ has $r$-fellow travel property (see Definition \ref{Fellow}). 
        
        By Corollary~\ref{ProjMidSet}, there exists $N = N(\tau_0, C, M)$ such that $$\proj_Y(o, u o) < N \qquad \text{and} \qquad\proj_Y(v o, ho)<N.$$
        Since $u o, v o\in Y$, we have $$\proj_Y(u o, v o) \ge d(uo,vo) >L.$$
        By triangle inequality of $\proj_Y$, we have 
        $$\proj_Y(o, ho)
        \ge \proj_Y(uo, vo) - \proj_Y(o,uo) - \proj_Y(vo,ho)
        > L - 2N.$$ 

    \end{proof}
    We now prove the upper bound for the other case.
    \begin{lem}\label{ProjHasNo}
        Assume $k\notin I$.  Then $\proj_Y(o, h o) < 2N.$
    \end{lem}
    \begin{proof}
        \begin{figure}[h]
            \centering
            \tikzset{every picture/.style={line width=0.75pt}} 

            \tikzset{every picture/.style={line width=0.75pt}} 
            
            \begin{tikzpicture}[x=0.75pt,y=0.75pt,yscale=-1,xscale=1]
            
            \draw    (34,127.05) -- (34,82.05) ;
            \draw [shift={(34,80.05)}, rotate = 90] [color={rgb, 255:red, 0; green, 0; blue, 0 }  ][line width=0.75]    (10.93,-3.29) .. controls (6.95,-1.4) and (3.31,-0.3) .. (0,0) .. controls (3.31,0.3) and (6.95,1.4) .. (10.93,3.29)   [fill={rgb, 255:red, 0; green, 0; blue, 0 }  ] (0, 0) circle [x radius= 3.35, y radius= 3.35];
            \draw [shift={(34,125.05)}, rotate = 10.14] [color={rgb, 255:red, 0; green, 0; blue, 0 }  ][fill={rgb, 255:red, 0; green, 0; blue, 0 }  ][line width=0.75]      (0, 0) circle [x radius= 3.35, y radius= 3.35]   ;
            \draw    (34,80.05) -- (84,80.05) ;
            \draw [shift={(86,80.05)}, rotate = 180] [color={rgb, 255:red, 0; green, 0; blue, 0 }  ][line width=0.75]    (10.93,-3.29) .. controls (6.95,-1.4) and (3.31,-0.3) .. (0,0) .. controls (3.31,0.3) and (6.95,1.4) .. (10.93,3.29)   [fill={rgb, 255:red, 0; green, 0; blue, 0 }  ] (0, 0) circle [x radius= 3.35, y radius= 3.35];
            \draw    (86,80.05) -- (86,125.05) ;
            \draw [shift={(86,127.05)}, rotate = 270] [color={rgb, 255:red, 0; green, 0; blue, 0 }  ][line width=0.75]    (10.93,-3.29) .. controls (6.95,-1.4) and (3.31,-0.3) .. (0,0) .. controls (3.31,0.3) and (6.95,1.4) .. (10.93,3.29)   [fill={rgb, 255:red, 0; green, 0; blue, 0 }  ] (0, 0) circle [x radius= 3.35, y radius= 3.35];
            \draw    (86,127.05) -- (140,127.05) ;
            \draw [shift={(142,127.05)}, rotate = 180] [color={rgb, 255:red, 0; green, 0; blue, 0 }  ][line width=0.75]    (10.93,-3.29) .. controls (6.95,-1.4) and (3.31,-0.3) .. (0,0) .. controls (3.31,0.3) and (6.95,1.4) .. (10.93,3.29)   [fill={rgb, 255:red, 0; green, 0; blue, 0 }  ] (0, 0) circle [x radius= 3.35, y radius= 3.35];
            \draw    (203,127.05) -- (252,127.05) ;
            \draw [shift={(254,127.05)}, rotate = 180] [color={rgb, 255:red, 0; green, 0; blue, 0 }  ][line width=0.75]    (10.93,-3.29) .. controls (6.95,-1.4) and (3.31,-0.3) .. (0,0) .. controls (3.31,0.3) and (6.95,1.4) .. (10.93,3.29)   [fill={rgb, 255:red, 0; green, 0; blue, 0 }  ] (0, 0) circle [x radius= 3.35, y radius= 3.35];
            \draw [shift={(201,127.05)}, rotate = 10.14] [color={rgb, 255:red, 0; green, 0; blue, 0 }  ][fill={rgb, 255:red, 0; green, 0; blue, 0 }  ][line width=0.75]      (0, 0) circle [x radius= 3.35, y radius= 3.35]   ;
            \draw    (254,127.05) -- (266.6,65.01) ;
            \draw [shift={(267,63.05)}, rotate = 101.48] [color={rgb, 255:red, 0; green, 0; blue, 0 }  ][line width=0.75]    (10.93,-3.29) .. controls (6.95,-1.4) and (3.31,-0.3) .. (0,0) .. controls (3.31,0.3) and (6.95,1.4) .. (10.93,3.29)   [fill={rgb, 255:red, 0; green, 0; blue, 0 }  ] (0, 0) circle [x radius= 3.35, y radius= 3.35];
            \draw    (267,63.05) -- (316.28,55.93) ;
            \draw [shift={(318.26,55.64)}, rotate = 171.78] [color={rgb, 255:red, 0; green, 0; blue, 0 }  ][line width=0.75]    (10.93,-3.29) .. controls (6.95,-1.4) and (3.31,-0.3) .. (0,0) .. controls (3.31,0.3) and (6.95,1.4) .. (10.93,3.29)   [fill={rgb, 255:red, 0; green, 0; blue, 0 }  ] (0, 0) circle [x radius= 3.35, y radius= 3.35];
            \draw    (318.26,55.64) -- (356.83,109.42) ;
            \draw [shift={(358,111.05)}, rotate = 234.35] [color={rgb, 255:red, 0; green, 0; blue, 0 }  ][line width=0.75]    (10.93,-3.29) .. controls (6.95,-1.4) and (3.31,-0.3) .. (0,0) .. controls (3.31,0.3) and (6.95,1.4) .. (10.93,3.29)   [fill={rgb, 255:red, 0; green, 0; blue, 0 }  ] (0, 0) circle [x radius= 3.35, y radius= 3.35];
            \draw  [color={rgb, 255:red, 0; green, 0; blue, 255 }  ] (358,111.05) -- (407,103.05) ;
            \draw [shift={(407,103.05)}, rotate = 169.69] [color={rgb, 255:red, 0; green, 0; blue, 255 }  ][line width=0.75]    (10.93,-3.29) .. controls (6.95,-1.4) and (3.31,-0.3) .. (0,0) .. controls (3.31,0.3) and (6.95,1.4) .. (10.93,3.29)   [fill={rgb, 255:red, 0; green, 0; blue, 255 }  ] (0, 0) circle [x radius= 3.35, y radius= 3.35];
            \draw   [color={rgb, 255:red, 0; green, 0; blue, 255 }  ] (407,103.05) -- (400.05,54.96) ;
            \draw [shift={(399.77,52.98)}, rotate = 81.78] [color={rgb, 255:red, 0; green, 0; blue, 255 }  ][line width=0.75]    (10.93,-3.29) .. controls (6.95,-1.4) and (3.31,-0.3) .. (0,0) .. controls (3.31,0.3) and (6.95,1.4) .. (10.93,3.29)   [fill={rgb, 255:red, 0; green, 0; blue, 255 }  ] (0, 0) circle [x radius= 3.35, y radius= 3.35];
            \draw   [color={rgb, 255:red, 0; green, 0; blue, 255 }  ] (399.77,52.98) -- (449.05,45.86) ;
            \draw [shift={(451.03,45.58)}, rotate = 171.78] [color={rgb, 255:red, 0; green, 0; blue, 255 }  ][line width=0.75]    (10.93,-3.29) .. controls (6.95,-1.4) and (3.31,-0.3) .. (0,0) .. controls (3.31,0.3) and (6.95,1.4) .. (10.93,3.29)   [fill={rgb, 255:red, 0; green, 0; blue, 255 }  ] (0, 0) circle [x radius= 3.35, y radius= 3.35];
            \draw   [color={rgb, 255:red, 0; green, 0; blue, 255 }  ] (451.03,45.58) -- (461.98,92.66) ;
            \draw [shift={(462.26,94.64)}, rotate = 257.78] [color={rgb, 255:red, 0; green, 0; blue, 255 }  ][line width=0.75]    (10.93,-3.29) .. controls (6.95,-1.4) and (3.31,-0.3) .. (0,0) .. controls (3.31,0.3) and (6.95,1.4) .. (10.93,3.29)   [fill={rgb, 255:red, 0; green, 0; blue, 255 }  ] (0, 0) circle [x radius= 3.35, y radius= 3.35];
            \draw   [color={rgb, 255:red, 0; green, 0; blue, 255 }  ] (517.26,83.64) -- (568.03,74.4) ;
            \draw [shift={(570,74.05)}, rotate = 169.69] [color={rgb, 255:red, 0; green, 0; blue, 255 }  ][line width=0.75]    (10.93,-3.29) .. controls (6.95,-1.4) and (3.31,-0.3) .. (0,0) .. controls (3.31,0.3) and (6.95,1.4) .. (10.93,3.29)   [fill={rgb, 255:red, 0; green, 0; blue, 255 }  ] (0, 0) circle [x radius= 3.35, y radius= 3.35];
            \draw [shift={(515.29,83.99)}, rotate = 10.14] [color={rgb, 255:red, 0; green, 0; blue, 255 }  ][fill={rgb, 255:red, 0; green, 0; blue, 255 }  ][line width=0.75]      (0, 0) circle [x radius= 3.35, y radius= 3.35]   ;
            \draw   [color={rgb, 255:red, 255; green, 0; blue, 0 }  ] (254,127.05) -- (295.57,167.65) ;
            \draw [shift={(297,169.05)}, rotate = 224.33] [color={rgb, 255:red, 255; green, 0; blue, 0 }  ][line width=0.75]    (10.93,-3.29) .. controls (6.95,-1.4) and (3.31,-0.3) .. (0,0) .. controls (3.31,0.3) and (6.95,1.4) .. (10.93,3.29)   [fill={rgb, 255:red, 255; green, 0; blue, 0 }  ] (0, 0) circle [x radius= 3.35, y radius= 3.35];
            \draw   [color={rgb, 255:red, 255; green, 0; blue, 0 }  ] (297,169.05) -- (326.6,140.58) ;
            \draw [shift={(328.04,139.19)}, rotate = 136.12] [color={rgb, 255:red, 255; green, 0; blue, 0 }  ][line width=0.75]    (10.93,-3.29) .. controls (6.95,-1.4) and (3.31,-0.3) .. (0,0) .. controls (3.31,0.3) and (6.95,1.4) .. (10.93,3.29)   [fill={rgb, 255:red, 255; green, 0; blue, 0 }  ] (0, 0) circle [x radius= 3.35, y radius= 3.35];
            \draw   [color={rgb, 255:red, 255; green, 0; blue, 0 }  ] (328.04,139.19) -- (359.67,174.56) ;
            \draw [shift={(361,176.05)}, rotate = 228.2] [color={rgb, 255:red, 255; green, 0; blue, 0 }  ][line width=0.75]    (10.93,-3.29) .. controls (6.95,-1.4) and (3.31,-0.3) .. (0,0) .. controls (3.31,0.3) and (6.95,1.4) .. (10.93,3.29)   [fill={rgb, 255:red, 255; green, 0; blue, 0 }  ] (0, 0) circle [x radius= 3.35, y radius= 3.35];
            \draw   [color={rgb, 255:red, 255; green, 0; blue, 0 }  ] (361,176.05) -- (331.37,207.59) ;
            \draw [shift={(330,209.05)}, rotate = 313.21] [color={rgb, 255:red, 255; green, 0; blue, 0 }  ][line width=0.75]    (10.93,-3.29) .. controls (6.95,-1.4) and (3.31,-0.3) .. (0,0) .. controls (3.31,0.3) and (6.95,1.4) .. (10.93,3.29)   [fill={rgb, 255:red, 255; green, 0; blue, 0 }  ] (0, 0) circle [x radius= 3.35, y radius= 3.35];
            \draw    [color={rgb, 255:red, 255; green, 0; blue, 0 }  ](370.86,237.87) -- (414.49,275.74) ;
            \draw [shift={(416,277.05)}, rotate = 220.96] [color={rgb, 255:red, 255; green, 0; blue, 0 }  ][line width=0.75]    (10.93,-3.29) .. controls (6.95,-1.4) and (3.31,-0.3) .. (0,0) .. controls (3.31,0.3) and (6.95,1.4) .. (10.93,3.29)   [fill={rgb, 255:red, 255; green, 0; blue, 0 }  ] (0, 0) circle [x radius= 3.35, y radius= 3.35];
            \draw [shift={(369.35,236.56)}, rotate = 10.14] [color={rgb, 255:red, 255; green, 0; blue, 0 }  ][fill={rgb, 255:red, 255; green, 0; blue, 0 }  ][line width=0.75]      (0, 0) circle [x radius= 3.35, y radius= 3.35]   ;
            \draw[dashed]  [color={rgb, 255:red, 0; green, 0; blue, 255 }  ] (351.57,62.96) .. controls (350.63,54.31) and (356.89,46.53) .. (365.54,45.59) -- (551.85,25.42) .. controls (560.5,24.48) and (568.28,30.74) .. (569.22,39.4) -- (574.31,86.41) .. controls (575.24,95.07) and (568.99,102.84) .. (560.33,103.78) -- (374.03,123.95) .. controls (365.37,124.89) and (357.6,118.63) .. (356.66,109.98) -- cycle ;
            \draw[dashed]  [color={rgb, 255:red, 255; green, 0; blue, 0 }  ] (274.78,104.5) .. controls (280.41,97.86) and (290.36,97.04) .. (297,102.67) -- (442.96,226.37) .. controls (449.6,232) and (450.43,241.95) .. (444.8,248.59) -- (414.22,284.67) .. controls (408.59,291.31) and (398.65,292.13) .. (392,286.5) -- (246.04,162.8) .. controls (239.4,157.17) and (238.58,147.23) .. (244.21,140.58) -- cycle ;
            \draw    (202,156.05) -- (241.24,134.99) ;
            \draw [shift={(243,134.05)}, rotate = 151.78] [color={rgb, 255:red, 0; green, 0; blue, 0 }  ][line width=0.75]    (10.93,-3.29) .. controls (6.95,-1.4) and (3.31,-0.3) .. (0,0) .. controls (3.31,0.3) and (6.95,1.4) .. (10.93,3.29)   ;
            \draw    (381,173.05) .. controls (441.09,165.17) and (447.81,143.7) .. (464.24,115.35) ;
            \draw [shift={(465,114.05)}, rotate = 120.38] [color={rgb, 255:red, 0; green, 0; blue, 0 }  ][line width=0.75]    (10.93,-3.29) .. controls (6.95,-1.4) and (3.31,-0.3) .. (0,0) .. controls (3.31,0.3) and (6.95,1.4) .. (10.93,3.29)   ;
            \draw[dashed]  [color={rgb, 255:red, 255; green, 0; blue, 0 }  ] (243.34,91.85) .. controls (246.64,74.13) and (256.76,61.14) .. (265.94,62.85) .. controls (275.13,64.56) and (279.9,80.32) .. (276.6,98.04) .. controls (273.31,115.77) and (263.19,128.76) .. (254,127.05) .. controls (244.81,125.34) and (240.04,109.58) .. (243.34,91.85) -- cycle ;
            \draw[dashed]  [color={rgb, 255:red, 0; green, 0; blue, 255 }  ] (322.9,92.67) .. controls (312.6,77.46) and (310.53,60.88) .. (318.26,55.64) .. controls (326,50.4) and (340.62,58.49) .. (350.92,73.7) .. controls (361.22,88.91) and (363.29,105.49) .. (355.56,110.72) .. controls (347.82,115.96) and (333.2,107.88) .. (322.9,92.67) -- cycle ;
            \draw    (282,91.05) .. controls (296.4,92.97) and (300.66,93.04) .. (309.83,82.43) ;
            \draw [shift={(311,81.05)}, rotate = 129.81] [color={rgb, 255:red, 0; green, 0; blue, 0 }  ][line width=0.75]    (10.93,-3.29) .. controls (6.95,-1.4) and (3.31,-0.3) .. (0,0) .. controls (3.31,0.3) and (6.95,1.4) .. (10.93,3.29)   ;
            \draw    (351,102.05) -- (363,99.05) ;
            \draw    (363,99.05) -- (367,109.05) ;
            \draw    (266,124.05) -- (262,134.05) ;
            \draw    (256,119) -- (266,124.05) ;
            
            \draw (25,128.4) node [anchor=north west][inner sep=0.75pt]    {$o$};
            \draw (15,94.4) node [anchor=north west][inner sep=0.75pt]    {$f_{0}$};
            \draw (47,59.4) node [anchor=north west][inner sep=0.75pt]    {$p_{0}$};
            \draw (86,89.4) node [anchor=north west][inner sep=0.75pt]    {$f_{0}^{-1}$};
            \draw (99,131.4) node [anchor=north west][inner sep=0.75pt]    {$s_{1}$};
            \draw (207,109.45) node [anchor=north west][inner sep=0.75pt]    {$s_{k}$};
            \draw (179,154.4) node [anchor=north west][inner sep=0.75pt]    {$h_{k} o$};
            \draw (259,82.4) node [anchor=north west][inner sep=0.75pt]    {$f_{k}$};
            \draw (279,43.4) node [anchor=north west][inner sep=0.75pt]    {$p_{k}$};
            \draw (322,81.4) node [anchor=north west][inner sep=0.75pt]    {$f_{k}^{-1}$};
            \draw (372,65.4) node [anchor=north west][inner sep=0.75pt]    {$f_{k+1}$};
            \draw (412,52.4) node [anchor=north west][inner sep=0.75pt]    {$p_{k+1}$};
            \draw (378,107.4) node [anchor=north west][inner sep=0.75pt]    {$s_{k+1}$};
            \draw (480.08,87.51) node [anchor=north west][inner sep=0.75pt]    {$\cdots $};
            \draw (456,56.4) node [anchor=north west][inner sep=0.75pt]    {$f_{k+1}^{-1}$};
            \draw (530,65.4) node [anchor=north west][inner sep=0.75pt]    {$s_{m}$};
            \draw (579,60.4) node [anchor=north west][inner sep=0.75pt]    {$\Phi _{g}( I') o$};
            \draw (425,273.4) node [anchor=north west][inner sep=0.75pt]    {$\Phi _{g}( I) o$};
            \draw (343.04,217.96) node [anchor=north west][inner sep=0.75pt]  [rotate=-39.1]  {$\cdots $};
            \draw (265.49,140.27) node [anchor=north west][inner sep=0.75pt]  [rotate=-45]  {$s_{k+1}$};
            \draw (301.51,125.4) node [anchor=north west][inner sep=0.75pt]  [rotate=-45]  {$f_{k+1}$};
            \draw (334.69,148.71) node [anchor=north west][inner sep=0.75pt]  [rotate=-45]  {$p_{k+1}$};
            \draw (360.09,179.88) node [anchor=north west][inner sep=0.75pt]  [rotate=-45]  {$f_{k+1}^{-1}$};
            \draw (395.77,233.3) node [anchor=north west][inner sep=0.75pt]  [rotate=-40.09]  {$s_{m}$};
            \draw (437,151) node [anchor=north west][inner sep=0.75pt]   [align=left] {action of $\displaystyle h_kf_{k} p_{k} f_{k}^{-1} h_k^{-1} \ $};
            \draw (205,58.4) node [anchor=north west][inner sep=0.75pt]    {$Y=Y_{k}$};
            \draw (330,39.4) node [anchor=north west][inner sep=0.75pt]    {$Y'$};
            \draw (163,124.4) node [anchor=north west][inner sep=0.75pt]    {$\cdots $};

            \end{tikzpicture}
                \caption{Proof of Lemma \ref{ProjHasNo}}
            \label{fig:ProfProjHasNo}
        \end{figure}
        To achieve the bound, we consider the element $h':=\Phi_g(I')$ for $I': = I\cup \{k\}$ and its associated  admissible path  $\gamma': = \gamma_g(I')$. Recalling 
        $$h=\Phi_g(I) = \prod_{j=0}^{m-1} (f_j p_j f_j^{-1})^{\epsilon_j} s_{j+1}\,\text{ and }\, h_k = \prod_{j=0}^{k-1} (f_{j} p_{j} f_{j}^{-1})^{\epsilon_{j}} s_{j+1}$$ we have $$\begin{aligned}
        h'=\Phi_g(I') &= \left(\prod_{j=0}^{k-1} (f_j p_j f_j^{-1})^{\epsilon_j} s_{j+1}\right) \left(f_k p_k f_k^{-1}\right) s_{k+1} \left(\prod_{j=k+1}^{m-1} (f_j p_j f_j^{-1})^{\epsilon_j} s_{j+1}\right) \\
        &\\
        &= h_k f_k p_k f_k^{-1} h_k^{-1}\Phi_g(I).    
        \end{aligned}$$
        Denote $Y' = h_k f_k p_k f_k^{-1} \ax(f_k)$. Then $Y' = h_k f_k p_k f_k^{-1} h_k^{-1} Y$. See Figure \ref{fig:ProfProjHasNo}.
        
        
        By construction, $[h_k o , h_k f_k o]$ is the geodesic segment of the $(L,\tau)$-admissible path   $\gamma_g(I')$ with endpoints in the contracting subset $Y$ corresponding to $f_k$ (Definition \ref{AdmDef}). By Corollary \ref{ProjMidSet}, there exists $N = N(\tau_0, C, M) > 0$ such that $$\proj_Y(o, h_k o) < N.$$

        Similarly,  $[h_k f_k p_k o, h_k f_k p_k f_k^{-1} o]$ is the geodesic segment  with endpoints in the contracting subset $Y'$ corresponding to $f_k^{-1}$. Again by Corollary \ref{ProjMidSet}, $$\proj_{Y'}(h_k f_k p_k f_k^{-1} o, h' o) < N.$$
        Noting that $h' = h_k f_k p_k f_k^{-1} h_k^{-1} h$ and $Y' = h_k f_k p_k f_k^{-1} h_k^{-1} Y$, by the isometric action of $h_k f_k p_k^{-1} f_k^{-1} h_k^{-1}$ we derive $$\proj_Y(h_ko, ho) < N.$$ By triangle inequality of $\proj_Y$, we have $$\proj_Y(o, h o) \le \proj_Y(o, h_k o) + \proj_Y(h_k o, h o) < 2N$$
        completing the proof.
    \end{proof}
    With the two lemmas as above, we complete the proof of the injectivity of $\Phi_g$.

    \begin{proof}[Proof of Lemma \ref{PhiInj}]

    Towards contradiction, let $I\ne I'\in \str$ so that $\Phi_g(I) = \Phi_g(I')=:h$. Let $k$ be the smallest integer that lies in exactly one of $I$ and $I'$. Then $k$ is the first position where the map $\Phi_g$ inserts $(f_{k} p_{k} f_{k}^{-1})$ differently for $I$ and $I'$. 
    
    Without loss of generality, assume $k \in I$ and $k\not\in I'$. We identify $(\epsilon_0, \cdots, \epsilon_{k-1})$ with $I\cap[0,k-1] = I'\cap [0,k-1]$ and denote
    $$h_k = \prod_{j=0}^{k-1} (f_{j} p_{j} f_{j}^{-1})^{\epsilon_{j}} s_{j+1}.$$
    and $Y:= h_k \ax(f_k)$. 
    
        
On the one hand, since $k\in I$, Lemma~\ref{ProjHas} gives
\[
\proj_Y(o,ho)=\proj_Y\!\bigl(o,\Phi_g(I)o\bigr)>L-2N.
\]
On the other hand, since $k\notin I'$, Lemma~\ref{ProjHasNo} yields
\[
\proj_Y(o,ho)=\proj_Y\!\bigl(o,\Phi_g(I')o\bigr)<2N.
\]
This contradicts the assumption that $h=\Phi_g(I)=\Phi_g(I')$, since $L>4N$ by the choice of (L3) in \textsection \ref{ChoCon}. Hence $\Phi_g$ is injective.
    \end{proof}

    \subsection{Completion of  the proofs of  Theorem \ref{AGroTig} and Theorem \ref{MainThm}}
    \begin{proof}[Proof of Theorem \ref{AGroTig}]
To achieve the gap, it suffices to verify the three conditions in Lemma~\ref{CritGapLem} for $A:=A_{\theta,L,M}$.

Let $g\in A$. Since $g\in\mathcal{D}_M(\theta,L)$, it admits a  $6M$–almost geodesic decomposition $g=s_1\cdots s_m$ with $m\ge \theta d(o,go)$ by Lemma~\ref{lem:QuasiConvex->AlmostGeodesic}. This proves the condition (1).

Let $\mathcal{F}=\{fpf^{-1}:f\in F,\ p\in P\}$, which is finite. By Remark~\ref{rem:alterDef},
\[
\Phi_g(I)=\prod_{i=0}^{m-1}(f_ip_if_i^{-1})^{\epsilon_i}s_{i+1}
\]
with letters in $\mathcal{F}$, and Lemma~\ref{PhiInj} shows $\Phi_g$ is injective, giving (2).

Finally, (3) follows from Lemma~\ref{MaptoCoset}. Hence, $\omega_G >\omega_A$ and the ``Moreover" statement follow by Lemma~\ref{CritGapLem}.
\end{proof}

    We now  give the proof of Theorem \ref{MainThm} (restated below) using Theorem \ref{AGroTig}. 
    \begin{thm}
        Suppose $G\act X$ is a proper SCC action with contracting elements.  Let  $P$ be a finite non-degenerate subset (i.e. $P\cap E(G)=\emptyset$). Then there exists a constant $\omega_0<\omega_G$ with the following property. If  $H<G$ is a confined subgroup  with $P$ as confining subset, then $\omega_0 > \omega_{G/H}$. 
    \end{thm}
    \begin{proof}
        Let $[G/H]$ denote a complete set of right coset representatives of $H$ in $G$.

Let $M$ be as in Lemma~\ref{OutGrowTight}. By Theorem~\ref{AGroTig}, there exists $L=L(\tau_0,C,M)$ such that $A_{\theta,L,M}=[G/H]\cap \mathcal{D}_M(\theta,L)$
is growth tight for every $\theta\in (0,1]$.

By Lemma~\ref{OutGrowTight}, there exists $\theta>0$ such that $G\setminus \mathcal{D}_M(\theta,L)$ is growth tight.

Observe that
\[
[G/H]\subseteq A_{\theta,L,M}\cup\bigl(G\setminus \mathcal{D}_M(\theta,L)\bigr).
\]
Since both sets on the right-hand side are growth tight, it follows that $[G/H]$ is growth tight.

Finally, note that $\tau_0$ depends only on $P$, while $C$ and $M$ depend only on the action of $G$ on $X$. Hence $L$, $\theta$, and the resulting growth gap depend only on $P$.
   \end{proof}
     
\appendix
\section{Confined subgroups in Free groups (by Lihuang Ding and Kairui Liu)}

    We give two different proofs to the main theorem for confined subgroup in free groups. 

    Let $T_d$ be the regular tree of valence $d\ge 3$.
    \begin{thm}
    Let $\Gamma$ be a $d$-regular graph (allowing loops and multiple edges). Assume that there exists a finite number $R$ such that the following holds: For any subgraph of $\Gamma$ that is isometric to a subgraph of $T_{d}$, this subgraph has diameter at most $R$.
    
    Then $\omega_\Gamma<\log (d-1)$. \end{thm} 
    \begin{proof}
    Fix a basepoint $o\in \Gamma$. Consider the $n$-sphere $S_n:=\{v\in \Gamma: d(o,v)=n\}$ for $n\ge 1$.  Observe that any ball of radius $R+1$ in $\Gamma$ contains an embedded loop, so $$|S_{2R}|\le D:=(d-1)^{2R}-1.$$ Thus, $$|S_{n+2R}|\le |S_n|\cdot |S_{2R}|\le D|S_n|$$ and by induction, $|S_{2nR}|\le D^n$. There exists a constant $c>0$ depending on $R$ so that $|S_{n}|\le cD^{n/2R}$ for $n\ge 1$. Taking the limit shows $$\limsup_{n} \log |S_{n}|/{n}\le \frac{\log D}{2R}<\log (d-1)$$ The proof is complete.  
    \end{proof}

    \begin{cor}\label{ThmFree}
    The Schreier graphs associated to confined subgroups in free groups are growth tight.    
    \end{cor}
    \begin{proof}
    It is well-known that the space of rooted $d$-regular graphs is isomorphic to the space of rooted Schreier graphs associated to subgroups. The injectivity radius of the Schreier graph associated to a confined subgroup is bounded everywhere. We refer to \cite[Section 2]{Can20} for relevant discussion.   
    \end{proof}
    Here is an alternative   proof of the main theorem. 

    Let $B_n(r)$ be the ball of radius $r$ in the infinite valance-$(2n)$ tree for $n,r\ge 1$. Let $\Gamma$ be a graph and fix a basepoint of $\Gamma$. Define $Sh(a, n) = \{b\in \Gamma: d(e, b) = d(e, a) + d(a, b), d(a, b) = n\}$ for any $a\in \Gamma$, $n\ge 0$. Then $Sh(a, 0) = \{a\}$.
    \begin{lem}
         Let $n, m\ge 1$. Let $\Gamma$ be a graph with degree of vertices at most $2n$ and no subgraph isomorphic to $B_n(m)$. Then $|Sh(a, 2m)|\le (2n-1)^{2m} -1$ for any $a\in \Gamma$, $a\neq e$.
    \end{lem}
    \begin{proof}
        Since the metric $d$ is the graph metric, for any $a\in \Gamma$ and $a\neq e$, there exists a neighbor $b$ of $a$ with $d(e, b) = d(e, a) - 1$. Since the degree of $a$ is at most $2n$, we have $|Sh(a, 1)| \le 2n-1$.

        Let $b\in Sh(a, k)$. Since $d(e, b) = d(e, a) + d(a, b)$, there exists a geodesic segment segment $[e, b]$ passing through $a$. Since $d(a, b) = k$, denote the chosen geodesic segment by $\gamma = (e, \cdots, a, b_1, b_2, \cdots, b_k = b)$. Then $b\in Sh(b_{k-1}, 1)$ and $b_{k-1} \in Sh(a, k-1)$. Moreover, for any $c\in Sh(a, k-1)$, the number of $b$ with $b_{k-1} = c$ is at most $|Sh(c, 1)|\le 2n-1$ where the inequality holds since $c\neq e$. Thus $|Sh(a, k)| \le (2n-1)|Sh(a, k-1)|$. By induction, $|Sh(a, k)| \le (2n-1)^k$ for any $k\ge 1$.

        Suppose $|Sh(a, 2m)| = (2n-1)^{2m}$. Then each inequality above is taken as equal. So $|Sh(a, k)| = (2n-1)^k$ for $1\le k\le 2m$ and $|Sh(x, 1)| = 2n-1$ for $x\in Sh(a, k)$, $1\le k\le 2m$. Moreover, consider the subgraph $A$ induced by vertices $\bigcup_{i=0}^{2m} Sh(a, i)$. Then $A$ is a tree with degree $2n-1$ at $a$ and degree $2n$ elsewhere. Let $b\in Sh(a, m)$. Then the ball $B_A(b, m)$ is isomorphic to $B_n(m)$, which contradicts with that $\Gamma$ has no subgraph isomorphic to $B_n(m)$. So $|Sh(a, 2m)| \le (2n-1)^{2m} -1$.
    \end{proof}
    \begin{thm}
        Let $\Gamma$ be a graph with degree at most $2n$ and no subgraph isomorphic to $B_n(m)$. Let $\omega$ be the growth rate of $\Gamma$. Then $\omega < \log (2n-1)$.
    \end{thm}
    \begin{proof}
        Fix a basepoint $e\in \Gamma$. Let $S_n = \{a: d(a, e) = n\}$. Then for $n>2m$, $$S_{n} = \bigcup_{a\in S_{n-2m}} Sh(a, 2m).$$ By lemma above, $|Sh(a, 2m)| \le (2n-1)^{2m}-1$. So $|S_n| \le |S_{n - 2m}| ((2n-1)^{2m} -1)$. Since degree of $\Gamma$ is at most $2n$, we have $|S_{2m}| \le (2n)^{2m}$. Thus by induction, we have $$|S_{2mk}| \le ((2n-1)^{2m}-1)^{k-1} |S_{2m}| \le ((2n-1)^{2m}-1)^{k-1} (2n)^{2m}.$$ Let $\alpha = ((2n-1)^{2m}-1)^{\frac{1}{2m}}$. Then $\alpha < 2n-1$. Thus, $$\omega = \lim_{n\to \infty} \frac{\log|S_n|}{n} = \lim_{k\to \infty} \frac{\log|S_{2mk}|}{2mk} \le \lim_{k\to \infty} \frac{2m(k-1)\log \alpha + 2m \log{2n}}{2mk} = \log \alpha < \log(2n-1).$$
    \end{proof}
    Then by the same method we can prove Corollary \ref{ThmFree}.
	\bibliography{bibliography}
	\bibliographystyle{alpha}
\end{document}